\documentclass[10pt,a4paper]{scrartcl}

\usepackage{amssymb}
\usepackage{amsmath}
\usepackage{amsfonts}
\usepackage{bbm}
\usepackage{amsthm}
\usepackage{mathrsfs}	
\usepackage[hidelinks]{hyperref}\usepackage{color}
\usepackage[margin=2.5cm]{geometry}
\usepackage[all,cmtip]{xy}	
\usepackage[utf8]{inputenc}
\usepackage{graphicx}
\usepackage{varwidth}

\let\emph\undefined\newcommand{\emph}[1]{\textsl{#1}}
\usepackage{upgreek}
\usepackage{rotating}
\usepackage{tikz}
\usepackage{accents}
\usetikzlibrary{matrix,arrows,decorations.pathmorphing,shapes.geometric}
\usepackage{tikz-cd}
\usepackage{needspace}
\newcommand{\spaceplease}{\needspace{5\baselineskip}}

\usetikzlibrary{decorations.markings}
\usetikzlibrary{backgrounds}

\usepackage{etex}
\usetikzlibrary{svg.path}

\tikzstyle{tikzfig}=[baseline=-0.25em,scale=0.5]

\pgfkeys{/tikz/tikzit fill/.initial=0}
\pgfkeys{/tikz/tikzit draw/.initial=0}
\pgfkeys{/tikz/tikzit shape/.initial=0}
\pgfkeys{/tikz/tikzit category/.initial=0}

\pgfdeclarelayer{edgelayer}
\pgfdeclarelayer{nodelayer}
\pgfsetlayers{background,edgelayer,nodelayer,main}

\tikzstyle{none}=[inner sep=0mm]

\newcommand{\tikzfig}[1]{%
	{\tikzstyle{every picture}=[tikzfig]
		\IfFileExists{#1.tikz}
		{\input{#1.tikz}}
		{%
			\IfFileExists{./figures/#1.tikz}
			{\input{./figures/#1.tikz}}
			{\tikz[baseline=-0.5em]{\node[draw=red,font=\color{red},fill=red!10!white] {\textit{#1}};}}%
	}}%
}

\tikzstyle{every loop}=[]

\usepackage{tikzit}

\tikzstyle{black dot}=[fill=black, draw=black, shape=circle, minimum size=3pt, inner sep=0pt]
\tikzstyle{black dot small}=[fill=black, draw=black, shape=circle, minimum size=3pt, inner sep=0pt]
\tikzstyle{wbox}=[fill=white, draw=black, shape=rectangle, minimum height=0.5cm, minimum width=0.01cm]
\tikzstyle{bbox}=[fill=white, draw=blue, shape=rectangle, minimum height=0.5cm, minimum width=0.01cm]
\tikzstyle{rbox}=[fill=white, draw=red, shape=rectangle, minimum height=0.5cm, minimum width=0.01cm]
\tikzstyle{bwbox}=[draw=blue, shape=rectangle, minimum width=2cm, minimum height=0.5cm]
\tikzstyle{bbwbox}=[draw=blue, shape=rectangle, minimum width=1cm, minimum height=1cm]
\tikzstyle{big white circle}=[fill=white, draw=black, shape=circle, minimum width=0.75cm]
\tikzstyle{white dot big}=[fill=white, draw=black, shape=circle, inner sep=1pt]
\tikzstyle{white dot}=[fill=white, draw=black, shape=circle, minimum size=3pt, inner sep=0pt]
\tikzstyle{flat box}=[fill=white, draw=black, shape=rectangle, minimum width=1.3cm, minimum height=0.5cm,fill=morphismcolor]
\tikzstyle{square}=[fill=white, draw=black, shape=rectangle]
\tikzstyle{flat box 2}=[fill=white, draw=black, shape=rectangle, minimum height=0.5cm, minimum width=0.01cm,fill=morphismcolor]
\tikzstyle{bigbox}=[fill=white, draw=black, shape=rectangle, minimum height=0.5cm, minimum width=0.8cm,fill=morphismcolor]
\tikzstyle{over }=[front]
\tikzstyle{theta}=[fill=blue, draw=blue, shape=ellipse, minimum height=6pt, minimum width=6pt, inner sep=0pt]
\tikzstyle{thetabig}=[fill=blue, draw=blue, shape=ellipse, minimum width=1cm, minimum height=0.01cm]
\tikzstyle{thetainv}=[fill=blue, draw=red, shape=ellipse, minimum height=6pt, minimum width=6pt, inner sep=0pt]
\tikzstyle{thetabinv}=[fill=blue, draw=red, shape=ellipse, minimum width=1cm, minimum height=0.01cm]
\tikzstyle{bigdisk}=[draw=black, shape=circle, minimum width=3cm]
\tikzstyle{wdisk}=[shape=circle, minimum width=0.48cm,fill=white]
\tikzstyle{bigdisk2}=[draw=black, fill=lightgray, shape=circle, minimum width=3cm]
\tikzstyle{little disk}=[fill=white, draw=black, shape=circle, minimum width=0.5cm]

\tikzstyle{mid arrow}=[-, postaction={on each segment={mid arrow}}]
\tikzstyle{end arrow}=[->]
\tikzstyle{mover}=[-, link]
\tikzstyle{mydots}=[-,dotted]
\tikzstyle{open}=[-, line width=1pt,draw=blue,postaction={on each segment={mid arrow}}]
\tikzstyle{thick}=[-,line width=1pt]
\tikzstyle{dotarrow}=[->,dotted,draw=red,line width=1pt]
\tikzstyle{bdotarrow}=[->,dotted,draw=blue,line width=1pt]
\tikzstyle{red mid arrow}=[-, draw={rgb,255: red,214; green,42; blue,51}, postaction={on each segment={mid arrow}}, line width=1pt]
\tikzstyle{RED}=[-, draw={rgb,255: red,214; green,42; blue,51}]
\tikzstyle{REDdotted}=[-,dashed, draw={rgb,255: red,214; green,42; blue,51}]
\tikzstyle{reddots}=[-,dotted, draw={rgb,255: red,214; green,42; blue,51}]
\tikzstyle{bluedashed}=[-,dashed, draw=blue]
\tikzstyle{blue}=[-, draw=blue]
\tikzstyle{blue mid arrow}=[-, draw={rgb,255: red,23; green,37; blue,167}, postaction={on each segment={mid arrow}}, line width=1pt]
\tikzstyle{over}=[-, link]
\tikzstyle{mover}=[-, link]
\tikzstyle{mapsto}=[{|->}]

\usetikzlibrary{decorations.pathreplacing,decorations.markings}
\tikzset{
	on each segment/.style={
		decorate,
		decoration={
			show path construction,
			moveto code={},
			lineto code={
				\path [#1]
				(\tikzinputsegmentfirst) -- (\tikzinputsegmentlast);
			},
			curveto code={
				\path [#1] (\tikzinputsegmentfirst)
				.. controls
				(\tikzinputsegmentsupporta) and (\tikzinputsegmentsupportb)
				..
				(\tikzinputsegmentlast);
			},
			closepath code={
				\path [#1]
				(\tikzinputsegmentfirst) -- (\tikzinputsegmentlast);
			},
		},
	},
	mid arrow/.style={postaction={decorate,decoration={
				markings,
				mark=at position .7 with {\arrow[#1]{stealth}}
	}}},
}
\tikzset{%
	link/.style    = { white, double = black, line width = 1.8pt,
		double distance = 0.4pt },
	channel/.style = { white, double = black, line width = 0.8pt,
		double distance = 0.8pt },
}

\usepackage{mathtools}
\mathtoolsset{showonlyrefs}
\usepackage{epstopdf}
\usepackage{pdfpages}

\newtheoremstyle{mytheorem}
  {\topsep}
  {\topsep}
  {\slshape}
  {0pt}
  {\bfseries}
  {.}
  { }
  {\thmname{#1}\thmnumber{ #2}\thmnote{ {\normalfont\slshape(#3)}}}
  
  \newtheoremstyle{mydefinition}
    {\topsep}
    {\topsep}
    {\normalfont}
    {0pt}
    {\bfseries}
    {.}
    { }
    {\thmname{#1}\thmnumber{ #2}\thmnote{ {\normalfont\slshape(#3)}}}

\theoremstyle{mytheorem}
\newtheorem{theorem}{Theorem}[section]

\makeatletter
\newtheorem*{rep@theorem}{\rep@title}
\newcommand{\newreptheorem}[2]{%
	\newenvironment{rep#1}[1]{%
		\def\rep@title{#2 \ref{##1}}%
		\begin{rep@theorem}}%
		{\end{rep@theorem}}}
\makeatother

\newreptheorem{theorem}{Theorem}
\newreptheorem{corollary}{Corollary}
\newtheorem{lemma}[theorem]{Lemma}
\newtheorem{proposition}[theorem]{Proposition}
\newtheorem{corollary}[theorem]{Corollary}
\theoremstyle{mydefinition}
\newtheorem{definition}[theorem]{Definition}
\newtheorem{exx}[theorem]{Example}
\newenvironment{example}
{\pushQED{\qed}\exx}
{\popQED\endexx}
\newtheorem{remm}[theorem]{Remark}
\newenvironment{remark}
{\pushQED{\qed}\remm}
{\popQED\endremm}
\numberwithin{equation}{section}
\usepackage{enumitem}

\newenvironment{pnum}{\begin{enumerate}[topsep=2pt,parsep=2pt,partopsep=2pt,itemsep=0pt,label={(\roman{*})}]}{\end{enumerate}}

\DeclareMathSymbol{\Phiit}{\mathalpha}{letters}{"08}\let\Phi\undefined\newcommand{\Phi}{\Phiit}
\DeclareMathSymbol{\Psiit}{\mathalpha}{letters}{"09}\let\Psi\undefined\newcommand{\Psi}{\Psiit}
\DeclareMathSymbol{\Sigmait}{\mathalpha}{letters}{"06}\let\Sigma\undefined\newcommand{\Sigma}{\Sigmait}
\DeclareMathSymbol{\Xiit}{\mathalpha}{letters}{"04}
\DeclareMathSymbol{\Lambdait}{\mathalpha}{letters}{"03}\let\Lambda\undefined\newcommand{\Lambda}{\Lambdait}
\DeclareMathSymbol{\Piit}{\mathalpha}{letters}{"05}\let\Pi\undefined\newcommand{\Pi}{\Piit}
\DeclareMathSymbol{\Gammait}{\mathalpha}{letters}{"00}\let\Gamma\undefined\newcommand{\Gamma}{\Gammait}
\DeclareMathSymbol{\Omegait}{\mathalpha}{letters}{"0A}\let\Omega\undefined\newcommand{\Omega}{\Omegait}
\DeclareMathSymbol{\Upsilonit}{\mathalpha}{letters}{"07}\let\Upsilon\undefined\newcommand{\Upsilon}{\Upilonit}
\DeclareMathSymbol{\Thetait}{\mathalpha}{letters}{"02}\let\Theta\undefined\newcommand{\Theta}{\Thetait}

\def\Hom{\mathrm{Hom}}

\def\id{\mathrm{id}}
\def\SL{\operatorname{SL}}

\def\dim{\mathrm{dim}}
\let\to\undefined\newcommand{\to}{\longrightarrow}
\let\mapsto\undefined\newcommand{\mapsto}{\longmapsto}
\newcommand{\catf}[1]{\mathsf{#1}}
\newcommand{\Proj}{\operatorname{\catf{Proj}}}
\newcommand{\Mod}{\catf{Mod}}

\def\op{\mathrm{op}}

\usepackage[normalem]{ulem}
\newcommand{\naka}{\catf{N}^\catf{r}}

\newcommand{\ra}[1]{\ \xrightarrow{\ \  #1  \  \ }\ }
\def\Ch{\catf{Ch}_k}

\newcommand{\drinfeld}{\mathbb{D}}

\newcommand{\mfc}{\mathfrak{F}_{\cat{C}}}

\def\PBun{\catf{PBun}}

\newcommand{\trace}{\catf{t}}
\newcommand{\cat}[1]{\mathcal{#1}}

\newcommand{\flint}{\int_{\text{\normalfont f}\mathbb{L}}}
\newcommand{\alg}{\catf{A}}\newcommand{\coalg}{\catf{F}}\newcommand{\somealg}{\catf{T}}

\let\C\undefined\newcommand{\C}{\catf{C}}
\newcommand{\Surfc}{\catf{Surf}^\catf{c}}

\newcommand{\OC}{\catf{OC}}
\newcommand{\lint}{\int_\mathbb{L}}
\newcommand{\rint}{\int^\mathbb{R}}

\newcommand{\pF}{\mathbb{F}_\bullet}
\newcommand{\pA}{\mathbb{A}_\bullet}
\newcommand{\mfd}{\mathfrak{F}^\cat{C}}
\newcommand{\bspace}[2]{\left( \  #1 \ , \    #2  \ \right)}
\makeatletter
\newcommand{\ostar}{\mathbin{\mathpalette\make@circled\star}}
\newcommand{\make@circled}[2]{%
	\ooalign{$\m@th#1\smallbigcirc{#1}$\cr\hidewidth$\m@th#1#2$\hidewidth\cr}%
}
\newcommand{\smallbigcirc}[1]{%
	\vcenter{\hbox{\scalebox{0.77778}{$\m@th#1\bigcirc$}}}%
}
\makeatother
\newcommand{\mlabel}[1]{{\footnotesize $#1$}}

\definecolor{Blue}  {rgb} {0.282352,0.239215,0.803921}
\definecolor{Green} {rgb} {0.133333,0.545098,0.133333}
\definecolor{Red}   {rgb} {0.803921,0.000000,0.000000}
\definecolor{Violet}{rgb} {0.580392,0.000000,0.827450}
\definecolor{morphismcolor}{rgb}{1,1,1}

\setkomafont{caption}{\small\slshape}

\setkomafont{section}{\rmfamily\large}
\setkomafont{subsection}{\rmfamily}
\setkomafont{subsubsection}{\rmfamily}
\setkomafont{subparagraph}{\normalfont\scshape}

\newtheorem*{theorem*}{Theorem}
\newtheorem*{corollary*}{Corollary}

\newcommand{\dolph}{^{\includegraphics[width=0.4cm]{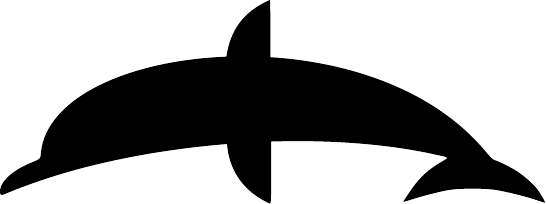}}}

\makeatletter
\renewcommand\section{\@startsection {section}{1}{\z@}%
	{-3.5ex \@plus -1ex \@minus -.2ex}%
	{2.3ex \@plus.2ex}%
	{\normalfont\scshape\centering}}
\makeatother

\usepackage{titletoc}
\dottedcontents{section}[1em]{\rmfamily}{1em}{0.2cm}
\dottedcontents{subsection}[0em]{}{3.3em}{1pc}

\begin{document} 
\vspace*{-10mm}
	\begin{flushright}
		\small
		{\sffamily [ZMP-HH/21-4]} \\
		\textsf{Hamburger Beiträge zur Mathematik Nr.~892}\\
		\textsf{CPH-GEOTOP-DNRF151}
	\end{flushright}
	
\vspace{8mm}
	
	\begin{center}
		\textbf{\large{The differential graded Verlinde Formula and the Deligne Conjecture}}\\
		\vspace{6mm}

{\large Christoph Schweigert ${}^a$ and  Lukas Woike ${}^b$}

\vspace{3mm}

\normalsize
{\slshape $^a$ Fachbereich Mathematik\\ Universit\"at Hamburg\\
	Bereich Algebra und Zahlentheorie\\
	Bundesstra\ss e 55\\  D-20146 Hamburg }

\vspace*{3mm}	

{\slshape $^b$ Institut for Matematiske Fag\\ K\o benhavns Universitet\\
	Universitetsparken 5 \\  DK-2100 K\o benhavn \O }

	\end{center}
\vspace*{1mm}
	\begin{abstract}\noindent 
A modular category $\cat{C}$ gives rise to a differential graded modular functor, i.e.\ a system of projective mapping class group representations on chain complexes. This differential graded modular functor assigns to the torus the Hochschild chain complex and, in the dual description, the Hochschild cochain complex of $\cat{C}$. On both complexes, the monoidal product of $\cat{C}$ induces the structure of an $E_2$-algebra, to which we refer as the \emph{differential graded Verlinde algebra}. At the same time, the modified trace induces on the tensor ideal of projective objects in $\cat{C}$ a Calabi-Yau structure so that the cyclic Deligne Conjecture endows the Hochschild cochain and chain complex of $\cat{C}$ with a second $E_2$-structure. Our main result is that the action of a specific element $S$ in the mapping class group of the torus transforms the differential graded Verlinde algebra into this second $E_2$-structure afforded by the Deligne Conjecture. This result is established for both the Hochschild chain and the Hochschild cochain complex of $\cat{C}$. In general, these two versions of the result are inequivalent. In the case of Hochschild chains, we obtain a block diagonalization of the Verlinde algebra through the action of the mapping class group element $S$. In the semisimple case, both results reduce to the Verlinde formula. In the non-semisimple case, we recover after restriction to zeroth (co)homology earlier proposals for non-semisimple generalizations of the Verlinde formula.
\end{abstract}

\tableofcontents

\normalsize
\section{Introduction and summary}
For any fusion category over an algebraically closed field $k$ of characteristic zero, the $k$-vector space spanned by the isomorphism classes $[x_0], [x_1] \dots,[x_n]$ of its simple objects becomes an associative and unital algebra by means of the monoidal product:
By semisimplicity, we have a decomposition $x_i \otimes x_j \cong \bigotimes_{\ell=0}^n N_{ij}^\ell x_\ell$ of $x_i\otimes x_j$ into a direct sum over the basis of simple objects, in which $x_\ell$ occurs with multiplicity $N_{ij}^\ell$, a non-negative integer. These fusion rules allow us to write the multiplication explicitly as
\begin{align} [x_i] \otimes [x_j] = \sum_{\ell=0}^n N_{ij}^\ell [x_\ell] \ . \label{eqnfusionmult} \end{align}
By a slight abuse of notation, the symbol $\otimes$ will also be used for the multiplication.
The class $[I]$ of the monoidal unit $I$ (which by convention is the zeroth object $x_0$ in the list of simple objects) is the unit of the multiplication. The resulting algebra is called the \emph{Verlinde algebra} of the fusion category. One can also see it as the linearized version of the Grothendieck ring or the $K_0$-ring of $\cat{C}$, see \cite[Section~4.5]{egno}. 

New tools for the 
computation 
of the fusion coefficients $N_{ij}^\ell$
become available 
if $\cat{C}$ is  a \emph{semisimple modular category},
i.e.\ if it additionally has a non-degenerate braiding and a ribbon structure (we recall the terminology in more detail in a moment, see page~\pageref{pageterminology}). Modular categories form an important class of categories in representation theory and conformal field theory \cite{mose88,turaev,klm,huangverlinde,huang,egno}.
In this case, the famous \emph{Verlinde formula}
conjectured by Verlinde
\cite{verlinde}
and proven by Moore and Seiberg \cite{mose88,mooreseiberg}, Cardy~\cite{cardy}, Witten~\cite{witten}
and Turaev~\cite{turaev}
expresses the fusion coefficients $N_{ij}^\ell$ via the \emph{$S$-matrix}, an invertible $(n+1)\times (n+1)$-matrix whose $(i,j)$-entry 
is given by the evaluation of the graphical calculus of $\cat{C}$ \cite{rt1,rt2,turaev} on the Hopf link labeled by the two simple objects $x_i$ and $x_j$:
\begin{align} S_{ij}:=\begin{tikzpicture}
	\begin{pgfonlayer}{nodelayer}
		\node [style=none] (0) at (-2, 0) {};
		\node [style=none] (1) at (-1.25, 0) {};
		\node [style=none] (2) at (-1.5, 0.5) {};
		\node [style=none] (3) at (-1.75, 0.5) {};
		\node [style=none] (4) at (-2, 1) {};
		\node [style=none] (5) at (-1.25, 1) {};
		\node [style=none] (6) at (-2.5, 0.5) {};
		\node [style=none] (7) at (-0.75, 0.5) {};
		\node [style=none] (8) at (-2.25, -0.5) {\mlabel{X_i}};
		\node [style=none] (9) at (-1, -0.5) {\mlabel{X_j}};
	\end{pgfonlayer}
	\begin{pgfonlayer}{edgelayer}
		\draw [bend right=315, looseness=0.75] (1.center) to (3.center);
		\draw [bend right=45] (2.center) to (4.center);
		\draw [style=over, bend right=45] (0.center) to (2.center);
		\draw [style=over, bend left=45] (3.center) to (5.center);
		\draw [in=180, out=90] (6.center) to (4.center);
		\draw [bend right=45] (6.center) to (0.center);
		\draw [bend right=45] (1.center) to (7.center);
		\draw [bend left=45] (5.center) to (7.center);
	\end{pgfonlayer}
\end{tikzpicture}\in k \label{eqndefsmatrixelements}
\end{align}
Now  the Verlinde formula asserts 
\begin{align} N_{ij}^\ell = \sum_{p=0}^n \frac{S_{ip}S_{jp}\left(S^{-1}\right)_{p \ell}}{S_{0p}} \in k \ . \label{eqnverlindeformula}\end{align}
It should be mentioned that~\eqref{eqnverlindeformula} is only one of several incarnations of the Verlinde formula.

The Verlinde formula~\eqref{eqnverlindeformula}
relies  on semisimplicity.
Nonetheless, a lot of the ingredients above can be given sense beyond semisimplicity such that aspects of the Verlinde formula still hold. Proposals in this direction have been  given
in  \cite{fhst,fgst,grv,gainutdinovrunkel}, see \cite{lo,glo,cgr} for examples of modular categories which are not semisimple.
One of the key differences between semisimple and non-semisimple finite tensor categories is that, in
the non-semisimple case, the homological algebra of tensor categories 
enriches 
the picture:
For instance, the (Hochschild) cohomology of finite tensor categories  has been  studied e.g.\ in \cite{gk,etingofostrik,schau,bichon,lq,negronplavnik}
(this refers mostly, but not exclusively to the Hopf algebraic  case).
Multiplicative structures 
have been investigated in
\cite{farinatisolotar,menichi,hermann}. 
More recently, the interaction of this homological algebra with low-dimensional topology has been developed in \cite{svea,dva,svea2,dmf}. 
The purpose of this article is to understand the content of the Verlinde formula 
within  a differential graded framework. This framework will 
feature the relevant quantities appearing in the homological 
algebra of a modular category and the higher structures that 
they naturally come equipped with.

Since this generalization can be best understood and proven as a topological result, it will be beneficial to recall the topological underpinning of the semisimple Verlinde formula. Indeed, a topological viewpoint already informed \cite{verlinde}. The viewpoint presented here is mostly due to \cite{witten,turaev}:
If $\cat{C}$ is a semisimple modular category, then $\cat{C}$ gives rise to a three-dimensional topological field theory $Z_\cat{C}$ by the Reshetikhin-Turaev construction \cite{rt1,rt2}.
In fact, semisimple modular categories are equivalent to once-extended three-dimensional topological field theories by a result of Bartlett, Douglas, Schommer-Pries and Vicary \cite{BDSPV15}.
The topological field theory $Z_\cat{C}$ assigns to the torus $\mathbb{T}^2$ the vector space
$Z_\cat{C}(\mathbb{T}^2)\cong k \left[\  [x_0] ,[x_1], \dots ,[x_n] \   \right]$ spanned by the isomorphism classes of simple objects of $\cat{C}$.
Since every mapping class group element can be seen as an invertible three-dimensional bordism,
the vector space $Z_\cat{C}(\mathbb{T}^2)$ comes with an action of the mapping class group $\SL(2,\mathbb{Z})$ of the torus
(generally, the mapping class group actions will be projective because of the \emph{framing anomaly}).
The multiplication~\eqref{eqnfusionmult} induced by the monoidal product
can be obtained by the evaluation of $Z_\cat{C}$ on the three-dimensional bordism 
\begin{align} P\times \mathbb{S}^1= \begin{tikzpicture}[scale=0.4]
	\begin{pgfonlayer}{nodelayer}
		\node [style=none] (0) at (1.5, 11.25) {};
		\node [style=none] (1) at (1.5, 9) {};
		\node [style=none] (2) at (1.5, 8) {};
		\node [style=none] (3) at (1.5, 5.75) {};
		\node [style=none] (4) at (4.75, 9.75) {};
		\node [style=none] (5) at (4.75, 7.5) {};
		\node [style=none] (6) at (0.5, 9.75) {};
		\node [style=none] (7) at (1, 10.25) {};
		\node [style=none] (8) at (2, 10.25) {};
		\node [style=none] (9) at (1.25, 10) {};
		\node [style=none] (10) at (1.75, 10) {};
		\node [style=none] (11) at (4.25, 8.75) {};
		\node [style=none] (12) at (5.25, 8.75) {};
		\node [style=none] (13) at (4.5, 8.5) {};
		\node [style=none] (14) at (5, 8.5) {};
		\node [style=none] (15) at (1, 7) {};
		\node [style=none] (16) at (2, 7) {};
		\node [style=none] (17) at (1.25, 6.75) {};
		\node [style=none] (18) at (1.75, 6.75) {};
	\end{pgfonlayer}
	\begin{pgfonlayer}{edgelayer}
		\draw [bend left=90, looseness=1.25] (0.center) to (1.center);
		\draw [bend right=90, looseness=1.25] (0.center) to (1.center);
		\draw [bend left=90, looseness=1.25] (2.center) to (3.center);
		\draw [bend right=90, looseness=1.25] (2.center) to (3.center);
		\draw [bend left=90, looseness=1.25] (4.center) to (5.center);
		\draw [bend right=90, looseness=1.25] (4.center) to (5.center);
		\draw [bend right=90, looseness=2.75] (2.center) to (1.center);
		\draw [in=180, out=0, looseness=0.50] (0.center) to (4.center);
		\draw [in=-180, out=0, looseness=0.75] (3.center) to (5.center);
		\draw [bend right=90, looseness=1.25] (7.center) to (8.center);
		\draw [bend left=75, looseness=1.75] (9.center) to (10.center);
		\draw [bend right=90, looseness=1.25] (11.center) to (12.center);
		\draw [bend left=75, looseness=1.75] (13.center) to (14.center);
		\draw [bend right=90, looseness=1.25] (15.center) to (16.center);
		\draw [bend left=75, looseness=1.75] (17.center) to (18.center);
	\end{pgfonlayer}
\end{tikzpicture} \ :\ \mathbb{T}^2 \sqcup\mathbb{T}^2\to\mathbb{T}^2  \, \label{eqntoruspop0}     \end{align}
where $P:\mathbb{S}^1 \sqcup \mathbb{S}^1\to \mathbb{S}^1$ is the two-dimensional pair of pants bordism. 
Note that this treats the two $\mathbb{S}^1$-factors of the torus differently: While on the first factor two copies of the circle are fused together via the pair of pants, the second factor is just a spectator.
As a result of treating the $\mathbb{S}^1$-factors differently, the multiplication $
Z_\cat{C}( P\times\mathbb{S}^1  ):Z_\cat{C}(\mathbb{T}^2)\otimes Z_\cat{C}(\mathbb{T}^2)\to Z_\cat{C}(\mathbb{T}^2)$ is maximally incompatible with the action of the mapping class group $\SL(2,\mathbb{Z})$ of the torus on the vector space $Z_\cat{C}(\mathbb{T}^2)$ meaning that, except for trivial cases, the mapping class group elements will never act through algebra morphisms.
More explicitly, if we pick a mapping class group element $R\in \SL(2,\mathbb{Z})$
and conjugate the multiplication with $R$, i.e.\ replace it with
$ Z_\cat{C}(R) \circ Z_\cat{C}( P\times\mathbb{S}^1  ) \circ \left(  Z_\cat{C}(R)^{-1}\otimes Z_\cat{C}(R) ^{-1}          \right)$,
the result
will generally be different from $Z_\cat{C}( P\times\mathbb{S}^1  )$. Phrased differently, the mapping class group orbit of $Z_\cat{C}( P\times\mathbb{S}^1  )$ is very non-trivial. 
Now the idea is to find within the mapping class group orbit of $Z_\cat{C}( P\times\mathbb{S}^1  )$ a multiplication which is as easy as possible, preferably diagonal. Then $Z_\cat{C}( P\times\mathbb{S}^1  )$ may be reconstructed from this easy multiplication and the mapping class group action on $Z_\cat{C}(\mathbb{T}^2)$.
Verlinde's formula,
when understood topologically,
tells us that this is indeed possible in the semisimple case:
To describe the solution, we identify a mapping class group element of the torus 
with the element in $\SL(2,\mathbb{Z})$
describing its action on the first homology
$H_1(\mathbb{T}^2;\mathbb{Z})\cong\mathbb{Z}^2$; 
it is important that here the `first' circle factor is exactly the `first' one from the definition of the multiplication, i.e.\ the one participating in the fusion.
Now consider the mapping class $ S:=\begin{pmatrix} 0 & -1 \\ 1 &\phantom{-}0\end{pmatrix}\in\SL(2,\mathbb{Z})$, 
the so-called \emph{$S$-transformation}.  If we conjugate the multiplication
$Z_\cat{C}( P\times\mathbb{S}^1  )$
with the automorphism $Z_\cat{C}(S)$ of $Z_\cat{C}(\mathbb{T}^2)$,
we transform the multiplication~\eqref{eqnfusionmult} coming from the monoidal product
into the very simple \emph{diagonal multiplication} which can be shown to be given by
\begin{align}
[x_i]\star [x_j] = \delta_{i,j} d_i^{-1}\cdot  [x_i] \ , \label{eqnstarprod0}
\end{align}
where $d_i=S_{i,0}=S_{0,i} \in k^\times$ is the \emph{quantum dimension} of $x_i$.  
In other words, the automorphism $Z_\cat{C}(S)$
diagonalizes the multiplication coming from the monoidal product. In yet another equivalent description, we may say that the map
\begin{align}
Z_\cat{C}(S) :      \bspace{Z_\cat{C}(\mathbb{T}^2)}{Z_\cat{C}( P\times\mathbb{S}^1  )} \ra{\cong} \bspace{Z_\cat{C}(\mathbb{T}^2)}{\star}\label{eqnverlindeisoalg}
\end{align}
is an isomorphism of algebras. In the canonical basis of $Z_\cat{C}(\mathbb{T}^2)$ given by the classes of simple objects, the matrix elements of the automorphism $Z_\cat{C}(S)$ turn out to be precisely the numbers $S_{ij}$ from \eqref{eqndefsmatrixelements}. If we use this matrix presentation of $Z_\cat{C}(S)$ and spell out what it means for $Z_\cat{C}(S)$ to be an algebra isomorphism of the form \eqref{eqnverlindeisoalg}, we arrive at the Verlinde formula~\eqref{eqnverlindeformula}.

When attempting to generalize the topological setup used to describe the Verlinde formula above to the non-semisimple case,
one faces
the problem that in order to build a once-extended three-dimensional topological field theory in the sense of \cite{rt1,rt2,BDSPV15}
from a modular category, semisimplicity is needed.  
If one is willing to give up the duality of the bordism category, the results in \cite{gai} 
generalize a substantial part of the Reshetikhin-Turaev construction to the non-semisimple case using work of Lyubashenko \cite{lyubacmp,lyuba,lyubalex}
and the theory of modified traces \cite{geerpmturaev,mtrace1,mtrace2,mtrace3,mtrace}. These constructions, however, are still insensitive to the homological algebra of the modular category and the higher structures associated with it (which is exactly what we  include in this article).
Fortunately, the structures actually needed to describe the topological setup above \emph{do} exist within a homotopy coherent framework, namely in terms of \emph{differential graded modular functors} instead of topological field theories, see \cite{tillmann,baki} for the definition of a modular functor with values in vector spaces.
A differential graded modular functor comes very close to a three-dimensional chain complex valued topological field theory, but cannot necessarily
be evaluated on non-invertible three-dimensional bordisms. In other words, a differential graded modular functor is an assignment of a chain complex (the so-called \emph{conformal block}) to each surface. These complexes will carry a homotopy coherent projective action of the respective mapping class groups and will satisfy excision, i.e.\ are compatible with gluing. 
In \cite{dmf}, it is proven, as an extension of \cite{svea,dva,svea2},
that any not necessarily semisimple modular category gives rise to a differential graded modular functor that in zeroth homology reduces to Lyubashenko's vector space valued modular functor \cite{lyubacmp,lyuba,lyubalex}.

In order to present our main results on the differential graded Verlinde algebra,
let us recall and fix some terminology: 
For a fixed field $k$, which will be assumed to be algebraically closed throughout the article (unlike for the discussion of the semisimple case above, we do not assume characteristic zero),       \label{pageterminology}
a \emph{finite category}  is an abelian category enriched over finite-dimensional $k$-vector spaces 
with enough projective objects and finitely many simple objects up to isomorphism; additionally,  we require that every object has finite length.
A \emph{tensor category} is a
linear abelian rigid monoidal category with simple unit.
A \emph{finite tensor category} in the sense of Etingof and Ostrik \cite{etingofostrik} is a tensor category with a finite category as underlying linear category. 
A finite tensor category $\cat{C}$ with a braiding, i.e.\ natural isomorphisms $c_{X,Y}:X\otimes Y\to Y\otimes X$ for $X,Y\in\cat{C}$ subject to several coherence conditions, is called a \emph{braided finite  tensor category}. 
From a topological viewpoint, the braiding extends the monoidal product to the structure of an algebra over the little disks operad $E_2$.
An extension to an algebra over the \emph{framed} little disks operad \cite{salvatorewahl} amounts to a \emph{balancing}, i.e.\ a natural automorphism of the identity whose components $\theta_X :X\to X$ satisfy
$
\theta_{X\otimes Y}=c_{Y,X}c_{X,Y}(\theta_X\otimes \theta_Y)$ for $X,Y\in\cat{C}$ and $\theta_I=\id_I$, where $I$ is the monoidal unit of $\cat{C}$.
A \emph{finite ribbon category}
is a braided finite tensor category $\cat{C}$ with balancing $\theta$ that is compatible with the duality $-^\vee$ in the sense that $\theta_{X^\vee}=\theta_X^\vee$ for $X\in\cat{C}$. 
The \emph{Müger center} of a
braided finite  tensor category $\cat{C}$ is the full subcategory of $\cat{C}$ given by the \emph{transparent objects}, i.e.\ those objects $X\in\cat{C}$ that satisfy $c_{Y,X}c_{X,Y}=\id_{X\otimes Y}$ for every $Y\in\cat{C}$.
The braiding  $c$ (and then also the braided finite tensor category) is referred to as \emph{non-degenerate}\label{nondegenpage} if its Müger center is as small as possible, namely spanned by the monoidal unit under finite direct sums.
A \emph{modular category} is a non-degenerate finite ribbon category. 	

The main result of \cite{dmf}
is that any modular category gives canonically rise to a differential graded modular functor, i.e.\ a symmetric monoidal functor	\begin{align}
\mfc\ :\ \cat{C}\text{-}\Surfc\to\Ch  \label{eqnmfcintro}
\end{align}
from (the central extension of) a category of extended surfaces, whose boundary components are labeled with projective objects in $\cat{C}$,
to the category of differential graded vector spaces over $k$ (one can also allow non-projective boundary labels). The differential graded modular functor $\mfc$ satisfies an excision property which allows us to compute the conformal block $\mfc(\Sigma,\underline{X})$ for a  surface $\Sigma$ with boundary label $\underline{X}$ via a pair of pants decomposition and a gluing procedure using homotopy coends. This is a consequence of the fact that, on a given fixed surface, the differential graded modular functor is constructed as a homotopy colimit over a contractible $\infty$-groupoid of colored markings; this is referred to as \emph{homotopy coherent Lego Teichmüller game} and an extension of the techniques used by Bakalov and Kirillov \cite{bakifm}, which, in turn, crucially rely on classical results on cut systems of surfaces due to Grothendieck \cite{grothendieck}, Hatcher and Thurston \cite{hatcherthurston} and Harer \cite{harer}.
On the closed torus, the differential graded modular functor produces the Hochschild complex of $\cat{C}$. More precisely,
the choice of a \emph{certain specific colored marking} on the torus gives us an equivalence $
\lint^{X\in\Proj\cat{C}}\cat{C}(X,X)\ra{\simeq}\mfc(\mathbb{T}^2)$
from the Hochschild complex $\lint^{X\in\Proj\cat{C}}\cat{C}(X,X)$ of $\cat{C}$ to the conformal block $\mfc(\mathbb{T}^2)$.
Recall that for a finite (tensor) category $\cat{C}$, the Hochschild complex is
the homotopy coend $\lint^{X\in\Proj \cat{C}} \cat{C}(X,X)$ running over the endomorphism spaces of projective objects. Explicitly, it is given by the (normalized) chains on the simplicial vector space
\begin{center}
	{	\begin{equation}\label{hheqn}
		\begin{tikzcd}
		& 	\displaystyle \dots 
		\ar[r, shift left=4] \ar[r, shift right=4] \ar[r] & \displaystyle \bigoplus_{X_0,X_1 \in \Proj\cat{C}}  \cat{C}(X_1,X_0)\otimes \cat{C}(X_0,X_1) \ar[r, shift left=2] \ar[r, shift right=2]
		\ar[l, shift left=2] \ar[l, shift right=2]
		& \displaystyle \bigoplus_{X_0 \in \Proj \cat{C}} \cat{C}(X_0,X_0)\ ,  \ar[l] \\
		\end{tikzcd}
		\end{equation}}
\end{center}
where $\cat{C}(-,-)$ denotes the morphism vector spaces.
When writing $\cat{C}$, as a linear category, as the category of finite-dimensional modules over a finite-dimensional algebra $A$, we recover  the Hochschild complex of $A$. This is a form of the \emph{Agreement Principle} of McCarthy \cite{mcarthy} and Keller \cite{keller}.

In \cite{dva} it was already established that the Hochschild chain complex of a finite braided tensor category comes with a non-unital $E_2$-multiplication generalizing the one discussed above in~\eqref{eqnfusionmult} for the semisimple case. 
Already in the setting of ordinary linear modular functors, it is a crucial idea
for the understanding of the Verlinde formula
to consider centers and class functions simultaneously. For an in-depth study of the multiplicative structure on the differential graded conformal block for the torus, this means that the Hochschild chain complex of $\cat{C}$ must be treated in tandem with the Hochschild \emph{cochain} complex of $\cat{C}$, i.e.\ the homotopy end
$\rint_{X\in \Proj \cat{C}} \cat{C}(X,X)$.
The latter is the value of the \emph{dual} differential graded modular functor $ \mfd:=\mfc^*$  on the torus, i.e.\ the functor obtained by taking point-wise  the dual chain complex in \eqref{eqnmfcintro}. The fact that this really yields the Hochschild cochain complex makes use of the Calabi-Yau structure on the tensor ideal $\Proj\cat{C}\subset \cat{C}$, see \cite[Remark~3.12]{dmf} and also \cite{trace}.
While the Hochschild chain and cochain complex of a modular category are degree-wise dual as chain complexes,
obtaining an $E_2$-structure on the Hochschild \emph{cochain} complex of a finite tensor category that is induced by the monoidal product is significantly more involved.
We prove the following result for the Hochschild cochain complex of a unimodular braided finite tensor category (unimodularity is implied by modularity):

\begin{reptheorem}{thmverlindecochain}
	Let $\cat{C}$ be a unimodular braided finite tensor category with chosen trivialization $D\cong I$ of the distinguished invertible object of $\cat{C}$.
	Then the Hochschild cochain complex $\rint_{X\in \Proj \cat{C}} \cat{C}(X,X)$ inherits from its braided monoidal product the structure of an $E_2$-algebra. 
\end{reptheorem}
We refer to this $E_2$-algebra as the \emph{differential graded Verlinde algebra} on the Hochschild cochain complex of $\cat{C}$ and denote the product by $\otimes$. 
If $\cat{C}$ is modular, we have,
by passing to the dual differential graded modular functor,
the homotopy coherent mapping class group action on $\rint_{X\in \Proj \cat{C}} \cat{C}(X,X)$ at our disposal. By acting with the mapping class group element $S=\begin{pmatrix} 0 & -1 \\ 1 &\phantom{-}0\end{pmatrix}\in \SL(2,\mathbb{Z})$, we obtain another multiplication --- ideally a simpler one which does not depend on the monoidal product.
This is exactly the idea behind the Verlinde formula 
in its formulation~\eqref{eqnverlindeisoalg}.
In fact, there is a natural candidate for an $E_2$-structure on the Hochschild cochain complex, which very conveniently does not see the monoidal product at all, but only the linear structure, namely the well-known $E_2$-structure  afforded by the 
\emph{Deligne Conjecture}: Deligne conjectured in 1993 that the Gerstenhaber structure on the Hochschild cohomology of an associative algebra \cite{gerstenhaber} has its origin in 
an $E_2$-structure on the Hochschild cochain complex of that algebra (for a suitable model of $E_2$). By now numerous proofs exist \cite{tamarkin,cluresmith,bergerfresse}, including proofs of the \emph{cyclic Deligne Conjecture} \cite{tradlerzeinalian,costellotcft,kaufmann}, a refinement for symmetric Frobenius algebras.

As our first main result, we prove that the $S$-transformation (or rather its inverse because of the dualization) indeed transforms the $E_2$-algebra induced by the monoidal product  (Theorem~\ref{thmverlindecochain}) into Deligne's $E_2$-structure.
This means that, as in the semisimple case,
the Verlinde algebra lies in the mapping class group orbit of a  simpler $E_2$-algebra structure that just uses the linear structure of $\cat{C}$.

\begin{reptheorem}{thmdgva1}[Verlinde formula for the Hochschild cochain complex]
	For any modular category $\cat{C}$, the action of the mapping class group element  $S^{-1}=\begin{pmatrix} \phantom{-}0 & 1 \\- 1 &0\end{pmatrix}\in \SL(2,\mathbb{Z})$   on the Hochschild cochain complex of $\cat{C}$  yields an equivalence
	\begin{align} \mfd(S^{-1}):    \bspace{  \rint_{X\in \Proj \cat{C}} \cat{C}(X,X)  }{\otimes} \simeq \bspace{ \rint_{X\in \Proj \cat{C}} \cat{C}(X,X)  }{\smile} \label{eqnSinvequiv}
	\end{align} 
	of $E_2$-algebras which are given as follows:\begin{itemize}
		\item On the left hand side, the $E_2$-structure is the differential graded Verlinde algebra on the Hochschild cochain complex induced by the monoidal product
		(Theorem~\ref{thmverlindecochain}).
		\item On the right hand side, the $E_2$-structure is the one afforded by Deligne's Conjecture with the underlying multiplication being the cup product $\smile$.
	\end{itemize}
\end{reptheorem}
The proof provides natural models of the homotopy end and the $E_2$-operad such that
$\mfd(S^{-1})$
is even an isomorphism of $E_2$-algebras.
Moreover, we prove that both $E_2$-algebras in Theorem~\ref{thmdgva1} naturally extend to framed $E_2$-algebras such that $\mfd(S^{-1})$ is an equivalence of framed $E_2$-algebras, see Corollary~\ref{corframed} for the definition of these framed $E_2$-structures and the precise statement.

The effect of $S$ on the non-unital $E_2$-structure on the Hochschild \emph{chain} complex from \cite{dva}
is different and the subject of our second main result:

\begin{reptheorem}{thmdgva2}[Verlinde formula for the Hochschild chain complex]
	For any modular category $\cat{C}$, the 
	action of the mapping class group element  $S=\begin{pmatrix} 0 & -1 \\ 1 &\phantom{-}0\end{pmatrix}\in \SL(2,\mathbb{Z})$ yields an equivalence
	\begin{align}  \mfc( S): \bspace{  \lint^{X\in \Proj \cat{C}} \cat{C}(X,X)  }{\otimes} \simeq \bspace{ \lint^{X\in \Proj \cat{C}} \cat{C}(X,X)  }{\star } 
	\end{align} 
	of non-unital $E_2$-algebras whose multiplication, up to homotopy, is concentrated in degree zero. 
	\begin{itemize}
		\item On the left hand side, the $E_2$-structure is the differential graded Verlinde algebra on the Hochschild chain complex induced the monoidal product \cite{dva}.
		
		\item On the right hand side, the non-unital $E_2$-structure is the almost trivial one that is a part of the cyclic version of Deligne's Conjecture applied to the Calabi-Yau structure coming from the modified trace
		on the tensor ideal of projective objects.
	\end{itemize}
\end{reptheorem}

The product $\star$ was defined and investigated in \cite{trace}
using the \emph{trace field theory} $\Phi_\cat{C}:\OC\to \Ch$, an open-closed topological conformal field theory 
that can be associated to  a finite tensor category and a suitable trivialization of the right Nakayama functor of $\cat{C}$ as right $\cat{C}$-module functor
relative to a pivotal structure on $\cat{C}$.
Therefore, we have  the following additional information on $\star$:
\begin{itemize}
	\item The product $\star$ is block diagonal \cite[Proposition~5.3]{trace}.
	Hence, Theorem~\ref{thmdgva2} implies that the $S$-transformation `block diagonalizes' the product $\otimes$.
	\item The $\star$-product of the identity morphisms $\id_P$ and $\id_Q$ of two projective objects $P$ and $Q$ is given, up to boundary, by the handle element $\xi_{P,Q}\in\cat{C}(P,P)$ of $\Phi_\cat{C}$ \cite[Theorem~5.6]{trace},
	\begin{align} \id_P \star \id_Q \simeq \xi_{P,Q} \ ,\label{starproduct}  \end{align}
	a certain central endomorphism $\xi_{P,Q}:P\to P$	whose modified trace is given by
	$\trace_P \xi_{P,Q} = \dim\,\cat{C}(P,Q)$. 	
	Hence, the modified traces of the handle elements recover the Cartan matrix of $\cat{C}$.
\end{itemize}
Formula~\eqref{starproduct} is a generalization of~\eqref{eqnstarprod0} as can be seen 
if $P$ is simple. Then $\xi_{P,Q}$ can be identified with a number and
\begin{align} \id_P\star\id_ Q \simeq \left(   d_P^\text{m}\right)^{-1} \dim  \, \cat{C}(P,Q) \cdot \id_P \ ,
\end{align} 
where $d_P^\text{m}\in k^\times$ is the modified dimension of $P$.
Therefore, the product $\star$ extracted from the cyclic Deligne Conjecture generalizes the product $\star$ from~\eqref{eqnstarprod0} to the non-semisimple case.

Having stated the two main results, 
we will now briefly highlight special cases and implications of the results:

\subparagraph{Restriction to zeroth (co)homology.}
Specializing
Theorem~\ref{thmdgva1} to zeroth cohomology
recovers the formula proposed and proven by Gainutdinov and Runkel \cite{grv} 
as a generalization of the Verlinde formula to the non-semisimple case (Corollary~\ref{corgrv}).
This formula features a complete system of the simple objects in $\cat{C}$ and multiplicities in Jordan-Hölder series, see Corollary~\ref{corgrv}.
However, the differential graded Verlinde algebra on the Hochschild cochain complex is significantly richer than its restriction to zeroth cohomology. In particular, its product and Gerstenhaber bracket are non-trivial, see Example~\ref{exbracketnonzero}. 
For Theorem~\ref{thmdgva2}, the situation is different. Here one only has a statement in zeroth homology.
It leads to a formula involving the fusion coefficients in the linearized $K_0$-ring of $\cat{C}$ (Corollary~\ref{corKzero}).

\subparagraph{Partial three-dimensional extension for differential graded modular functors.}
A priori, a modular functor is less than a three-dimensional topological field theory.
For the differential graded modular functor of a modular category, however,
we can give the following partial extension result:

\begin{repcorollary}{corhighgenus}
	The differential graded modular functor $\mfc$ associated to a modular category $\cat{C}$ 
	extends to three-dimensional oriented bordisms of the form $\Sigma \times\mathbb{S}^1: \left( \mathbb{T}^2\right)^{\sqcup p} \to \left( \mathbb{T}^2\right)^{\sqcup q}$, where $\Sigma :  \left( \mathbb{S}^1\right)^{\sqcup p} \to \left( \mathbb{S}^1\right)^{\sqcup q}$ is a compact oriented two-dimensional bordism such that every component of $\Sigma$ has at least one incoming boundary component. 
\end{repcorollary}

On the bordisms $P\times\mathbb{S}^1$ from~\eqref{eqntoruspop0} and its reversed version,
this extension is given by the product $\otimes$ from Theorem~\ref{thmdgva2} and the product from Theorem~\ref{thmverlindecochain} dualized via the Calabi-Yau structure, respectively. On the solid torus seen as bordism $\mathbb{T}^2\to \emptyset$, one obtains the modified trace precomposed with the $S$-transformation. An extension to the solid closed torus as bordism $\emptyset \to \mathbb{T}^2$ will generally not exist in the non-semisimple case (Remark~\ref{remfurtherextension}).
Based on Corollary~\ref{corhighgenus},
we formulate in Corollary~\ref{cordimred} a higher genus Verlinde formula in terms of the trace field theory of $\cat{C}$.

\subparagraph{Conventions.}
Plenty of key notions have already been defined in the introduction, and more will follow in the main text.
In this additional short list, we want to collect some more technical or notational conventions.

\begin{enumerate}[topsep=2pt,parsep=2pt,partopsep=2pt,itemsep=0pt,label={(\arabic{*})}]
	\item
	For the entire article, we work over a fixed algebraically closed field $k$. We do not assume that $k$ has characteristic zero.
	
	\item We use the notation
	$\Ch$ for 
	the symmetric monoidal category of chain complexes over $k$.
	Whenever needed, we equip it with its projective model structure in which weak equivalences (for short: equivalences) are quasi-isomorphisms and fibrations are degree-wise surjections.
	By a \emph{(canonical) equivalence} between chain complexes (notation $\simeq$ as opposed to the notation $\cong$ reserved for isomorphisms)
	we do not necessarily mean 
	a map in either direction, but also allow a (canonical) zigzag.

	\item
	We follow the duality conventions of \cite[Section~2.10]{egno}:
	For every object $X \in \cat{C}$
	in a \emph{rigid} monoidal category,  we denote  the  \emph{left dual}  by $X^\vee$
	(it comes with an evaluation $d_X:X^\vee \otimes X\to I$ and a coevaluation $b_X:I\to X\otimes X^\vee$),
	and the  \emph{right dual} by ${^\vee \! X}$ 
	(it comes with an evaluation $\widetilde d_X : X\otimes  {^\vee \! X} \to I$ and a coevaluation $\widetilde b_X : I \to {^\vee \! X}\otimes X$).
	Evaluation and coevaluation are subject to the usual zigzag identities.
	By left and right duality we obtain the natural adjunction isomorphisms
	\begin{align}
	\begin{array}{rclrcl}
	\cat{C}(X\otimes Y,Z)&\cong& \cat{C}(X,Z\otimes Y^\vee) \ , &
	\cat{C}(Y^\vee\otimes X,Z)&\cong& \cat{C}(X,Y\otimes Z) \ , \\
	\cat{C}(X\otimes {^\vee \! Y}, Z) &\cong& \cat{C}(X,Z\otimes Y)\ ,&
	\cat{C}(Y\otimes X,Z)&\cong& \cat{C}(X,{^\vee \! Y} \otimes Z)    
	\end{array}
	\end{align}for $X,Y,Z\in\cat{C}$.

	\item Any finite (tensor) category $\cat{C}$ is a module category over the symmetric monoidal category of finite-dimensional $k$-vector spaces.
	This means that we have a tensoring $V\otimes X\in \cat{C}$ for a finite-dimensional vector space $V$ and $X\in \cat{C}$ and also a powering $X^V=V^*\otimes X\in \cat{C}$. Here $V^*$ is the dual vector space of $V$.
	\label{convtensoring}
	
	\item
	For the definition of the $S$-matrix in~\eqref{eqndefsmatrixelements},
	we have already used the \emph{graphical calculus} for morphisms in (braided) monoidal categories, see e.g.\ \cite{kassel}. This graphical calculus will be used throughout the text whenever the corresponding computations in equations would become too complicated or hardly insightful.
	Objects are symbolized by vertical lines and the monoidal product by the juxtaposition of lines (the monoidal unit is the empty collection of lines). The braiding and inverse braiding are denoted by an overcrossing and undercrossing, respectively. 
	The evaluation and coevaluation are denoted by a cap and cup, respectively.
	The morphisms are always to be read from bottom to top. The composition is represented by vertical stacking.\label{congraphcalc}
	
\end{enumerate}

 \subparagraph{Acknowledgments.} 
We would like to thank 
Adrien Brochier, 
Lukas Müller,
Cris Negron,
Peter Schauenburg
and Nathalie Wahl
for helpful discussions. 

CS is supported by the Deutsche Forschungsgemeinschaft (DFG, German Research
Foundation) 
under Germany’s Excellence Strategy -- EXC 2121 ``Quantum Universe'' -- 390833306.
LW gratefully acknowledges support by 
the Danish National Research Foundation through the Copenhagen Centre for Geometry
and Topology (DNRF151)
and by the European Research Council (ERC) under the European Union's Horizon 2020 research and innovation programme (grant agreement No. 772960).

\spaceplease
\section{The Hochschild chain complex of a modular category as differential graded conformal block for the torus\label{secrevewidmf}}
It was explained in the introduction that, after the choice of an auxiliary datum, namely a  specific colored marking of the torus,
we obtain an equivalence $\lint^{X\in\Proj\cat{C}}\cat{C}(X,X)\ra{\simeq}\mathfrak{F}_\cat{C}(\mathbb{T}^2)$
from the Hochschild complex of $\cat{C}$ to the differential graded conformal block $\mathfrak{F}_\cat{C}(\mathbb{T}^2)$ of the torus.
The purpose of this section is to give a model of the conformal block for the torus which is closely related to  $\lint^{X\in\Proj\cat{C}} \cat{C}(X,X)$, but on which the effect of the $S$-transformation can be described in a more convenient way. This model is related to the one used in \cite{dva,dmf}, but we need to go  beyond that to later prove the main results
in later sections.

In any finite tensor category $\cat{C}$, one may define the \emph{canonical coend}
$ \mathbb{F}:=\int^{X\in\cat{C}}X^\vee \otimes X$
and the \emph{canonical end}
$
\mathbb{A}=\int_{X\in\cat{C}} X \otimes X ^\vee
$
which, due to their appearance
in \cite{lyubacmp,lyuba}, are also called the \emph{Lyubashenko coend and end},
respectively
(the duality conventions for the (co)end are not standard and vary between sources).
In \cite[Section~3.2]{dva} 
we introduced the \emph{(finite) homotopy coend}
$ \flint^{X \in \Proj \cat{C}} X^\vee \otimes X$
(the subscript `f' stands for `finite' because the coend can be reduced to finitely many projective objects). 
It is a differential graded object in $\cat{C}$ and can be defined by means of a simplicial object formally similar to the one used for the Hochschild complex.
This homotopy coend serves as a projective resolution of $\mathbb{F}$ that we use to express the Hochschild chain complex:

\begin{proposition}[$\text{\cite[Corollary~3.7 \& Theorem~3.9]{dva}}$] \label{corresocoend2}
	Let $\cat{C}$ be a pivotal finite tensor category.
	The (finite) homotopy coend
	$\flint^{X \in \Proj \cat{C}} X^\vee \otimes X$ is 
	a projective resolution of the canonical coend $\mathbb{F}=\int^{X\in\cat{C}} X^\vee\otimes X$ and  allows us to write the Hochschild complex of $\cat{C}$ up to equivalence as
	\begin{align} \label{eqnequivHH}\lint^{X\in\Proj \cat{C}} \cat{C}(X,X) \simeq \cat{C}\left( I, \flint^{X \in \Proj \cat{C}} X^\vee \otimes X  \right) \ ;  \end{align}
	in fact,  $\lint^{X\in\Proj \cat{C}} \cat{C}(X,X) \simeq \cat{C}\left( I, \mathbb{F}_\bullet  \right)$ 
	for any projection resolution $\mathbb{F}_\bullet$ of $\mathbb{F}$.
\end{proposition}

There is an analogue of Proposition~\ref{corresocoend2} for the \emph{Hochschild cochain complex}
$\rint_{X\in\Proj\cat{C}}\cat{C}(X,X)$
of a finite tensor category, 
i.e.\ the homotopy end over the subcategory of projective objects (which coincide with the injective ones).
Since, unlike the homotopy coend defining the Hochschild chain complex, homotopy \emph{ends}
were not recalled in the introduction, let us give a brief overview:
Let $\cat{A}$ be a linear category over $k$ (this will then be applied to $\cat{A}=\Proj\cat{C}$, where $\cat{C}$ is a finite tensor category).
Then its Hochschild cochain complex $\rint_{a\in \cat{A}} \cat{A}(a,a)$ 
is the homotopy end over the endomorphism spaces of objects in $\cat{A}$, i.e.\
the cochain complex of vector spaces which in cohomological degree $n\ge 0$ is given by
\begin{align}\label{eqnhochschildcomplex}
\left(\rint_{a\in \cat{A}} \cat{A}(a,a) \right)^n =\left\{ \begin{array}{cl}{\small \prod\limits_{a_0\in\cat{A}} \cat{A}(a_0,a_0)} & {\small\text{for $n=0$}} \ , \\ {\small\prod\limits_{a_0,\dots,a_n\in\cat{A}} \Hom_k     \left(    \cat{A}(a_1,a_0) \otimes \dots \otimes \cat{A}(a_n,a_{n-1}),\cat{A}(a_n,a_0)       \right)} & {\small \text{for $n\ge 1$}} \ .\end{array} \right.
\end{align}
The differential comes as usual from the composition in $\cat{A}$.
On the Hochschild cochain complex, one may define the \emph{cup product} $\smile$:
Let $\varphi$ and $\psi$ be a $p$-cochain and a $q$-cochain, respectively, and let
$(a_0,\dots,a_{p+q})$ be a $p+q$-tuple of objects in $\cat{A}$.
Then
the $(a_0,\dots,a_{p+q})$-component $(	\varphi \smile \psi)_{a_0,\dots,a_{p+q}}$ of the $p+q$-cochain $\varphi \smile \psi$ is given by
$
(	\varphi \smile \psi)_{a_0,\dots,a_{p+q}} := \varphi_{a_0,\dots,a_p} \circ_{a_p} \psi_{a_{p},\dots,a_{p+q}}
$,
where $\circ_{a_p}$ is the composition in $\cat{A}$ over $a_p$.

\begin{proposition}[$\text{\cite[Proposition~4.3]{E2} based on \cite{cartaneilenberg,bichon,shimizucoend}}$]\label{propcocomplex}	
	For any finite tensor category $\cat{C}$, there is a canonical equivalence $\rint_{X\in\Proj \cat{C}}    \cat{C}(X,X)   \simeq \cat{C}\left(I,\mathbb{A}^\bullet \right)$, where $\mathbb{A}^\bullet$ is an injective resolution of $\mathbb{A}$. In other words, there is an equivalence between the Hochschild cochain complex of $\cat{C}$ and the homotopy invariants of $\mathbb{A}$.
	For suitable models of the homotopy end and the injective resolution, respectively, the equivalence can be turned into an isomorphism.
\end{proposition}

For a finite tensor category $\cat{C}$, we denote by $Z(\cat{C})$ its \emph{Drinfeld center}, the \emph{braided} tensor category that consists of pairs of an object $X\in \cat{C}$ and a \emph{half braiding}, i.e.\ a natural isomorphism $X\otimes-\cong -\otimes X$ subject to coherence conditions, \cite[Section~7.13]{egno} for a textbook reference.
The Drinfeld center can be seen as the center of $\cat{C}$ as $E_1$-algebra  and is therefore an $E_2$-algebra, i.e.\ braided (and this braiding is actually the one that one can directly give based on the description of $Z(\cat{C})$ in terms of half braidings). 
It is also a finite tensor category, see \cite[Theorem~3.34]{etingofostrik} and \cite[Theorem~3.8]{shimizuunimodular}.
The forgetful functor $U:Z(\cat{C})\to \cat{C}$ is exact and therefore has a left adjoint $L:\cat{C}\to Z(\cat{C})$ and a right adjoint $R: \cat{C}\to Z(\cat{C})$. 
Since $U$ is strong monoidal, $L$ and $R$ are automatically oplax and lax monoidal, respectively, see \cite{bruguieresvirelizier,shimizuunimodular} for a more detailed account on the structure of these adjoint pairs and the (co)monads they give rise to. As a consequence, the images 
$\coalg:=LI$
and
$\alg:=RI$
of the monoidal unit $I \in \cat{C}$  
are a coalgebra and an algebra in $Z(\cat{C})$, respectively.  
The underlying objects in $\cat{C}$
\begin{align} U\coalg =\mathbb{F}=\int^{X\in\cat{C}} X^\vee \otimes X\ , \quad U\alg=\mathbb{A}=\int_{X\in\cat{C}} X\otimes X^\vee \label{eqnUcenter} \end{align}
are the canonical coend and the canonical end, respectively; their half braiding is often called the \emph{non-crossing half braiding}.

As a consequence of~\eqref{eqnUcenter},
$\mathbb{F}$ is coalgebra and $\mathbb{A}$ an algebra in $\cat{C}$.
In fact, the coalgebra structure $\delta : \mathbb{F}\to\mathbb{F}\otimes\mathbb{F}$ on $\mathbb{F}$ is induced by the coevaluations
$
X^\vee \otimes X\ra{X^\vee \otimes b_X \otimes X} X^\vee \otimes X\otimes X^\vee \otimes X $
while, dually, $\mathbb{A}$ inherits an algebra structure  $\gamma : \mathbb{A}\otimes \mathbb{A}\to\mathbb{A}$ on $\mathbb{A}$  induced by the evaluations 
$
X \otimes X^\vee \otimes X\otimes X^\vee \ra{X\otimes d_X\otimes X^\vee} X\otimes X^\vee$.

The left and the right adjoint to the forgetful functor $U$ are intimately related to the distinguished invertible object $D$ of $\cat{C}$  introduced in~\cite{eno-d}. More precisely, there exist canonical natural isomorphisms
$
L(D\otimes-) \cong R \cong L(-\otimes D)$ and $ R(D^{-1}\otimes -)\cong L \cong R (-\otimes D^{-1})   
$ \cite[Lemma~4.7 \& Theorem~4.10]{shimizuunimodular}.
A finite tensor category is called \emph{unimodular} if $D\cong I$. 
By the relation of $L$ and $R$, this is the case if and only if $L\cong R$, see again \cite[Theorem~4.10]{shimizuunimodular}.
This allows us to define the Radford map (we justify the terminology through the comments after Proposition~\ref{propSmatrix}):

\begin{definition}[Radford map]\label{defradfordmap}
	Let $\cat{C}$ be a finite tensor category. If $\cat{C}$ is unimodular
	and if a trivialization of $D$ is chosen (the possible choices form a $k^\times$-torsor), we define the $I$-component of the resulting natural isomorphism $UR\cong UL$ 
	as the \emph{Radford map}
	and denote it by
	$
	\Psi : \mathbb{A}=UR(I)\ra{\cong}UL(I)= \mathbb{F} $.
\end{definition}

A final ingredient is needed to describe the effect of the $S$-transformation explicitly:
For any finite braided tensor category $\cat{C}$, the maps $X^\vee \otimes X\to Y \otimes Y^\vee$ given by the double braiding
\begin{equation}\begin{tikzpicture}[scale=0.7]
	\begin{pgfonlayer}{nodelayer}
		\node [style=none] (10) at (2, -6.75) {\mlabel{X}};
		\node [style=none] (11) at (1, -6.75) {\mlabel{X^\vee}};
		\node [style=none] (12) at (4.25, -1.5) {\mlabel{Y^\vee}};
		\node [style=none] (13) at (3, -1.5) {\mlabel{Y}};
		\node [style=none] (14) at (2, -5) {};
		\node [style=none] (15) at (3, -5) {};
		\node [style=none] (16) at (2, -4) {};
		\node [style=none] (17) at (3, -4) {};
		\node [style=none] (18) at (2, -3) {};
		\node [style=none] (19) at (3, -3) {};
		\node [style=none] (20) at (1, -3) {};
		\node [style=none] (21) at (1, -5) {};
		\node [style=none] (22) at (4, -5) {};
		\node [style=none] (23) at (3, -2) {};
		\node [style=none] (24) at (4, -2) {};
		\node [style=none] (25) at (1, -6) {};
		\node [style=none] (26) at (2, -6) {};
	\end{pgfonlayer}
	\begin{pgfonlayer}{edgelayer}
		\draw [in=-90, out=90, looseness=0.75] (15.center) to (16.center);
		\draw [in=-90, out=90, looseness=0.75] (17.center) to (18.center);
		\draw [style=over, in=-90, out=90, looseness=0.75] (14.center) to (17.center);
		\draw [style=over, in=-90, out=90, looseness=0.75] (16.center) to (19.center);
		\draw [bend right=90, looseness=1.50] (18.center) to (20.center);
		\draw (20.center) to (21.center);
		\draw [bend right=90, looseness=1.25] (15.center) to (22.center);
		\draw [style] (21.center) to (25.center);
		\draw [style] (14.center) to (26.center);
		\draw [style] (23.center) to (19.center);
		\draw [style] (24.center) to (22.center);
	\end{pgfonlayer}
\end{tikzpicture}
\end{equation}\normalsize
induce a map $\drinfeld:\mathbb{F}\to\mathbb{A}$,
the so-called \emph{Drinfeld map} \cite{drinfeld}.
The Drinfeld map can be used to characterize non-degeneracy of the braiding (the definition was given on page~\pageref{nondegenpage} in the introduction):
By the main result of \cite{shimizumodular} a braided finite  tensor category is non-degenerate  if and only if its Drinfeld map is an isomorphism.
We may now explicitly describe the effect of the $S$-transformation on the differential 
graded conformal block of the torus as follows:

\begin{proposition}\label{propSmatrix}
	Let $\cat{C}$ be a modular category.
	After identification of the Hochschild complex  
	with $\cat{C}(I,\pF)$ for a projective resolution $\pF$ of $\mathbb{F}$ (Proposition~\ref{corresocoend2}),
	the mapping class group element
	$S=\begin{pmatrix} 0 & -1 \\ 1 &\phantom{-}0\end{pmatrix}$
	acts by the equivalence
	$
	\cat{C}(I,\pF) \ra{\drinfeld_\bullet}  	\cat{C}(I,\pA) \ra{\Psi_\bullet} \cat{C}(I,\pF)$, 
	where the first arrow is induced by the Drinfeld map 
	and the second arrow by the Radford map.
\end{proposition}

By \cite[Proposition~4.5]{eno-d} any modular category is unimodular, and we will tacitly assume in the sequel that an isomorphism $D\cong I$ has been fixed for any modular category (therefore, the Radford isomorphism is defined  here).
It is standard in the theory of modular functors that the $S$-transformation acts by a composition of (some form of the) Drinfeld and the Radford map, see e.g.\ \cite[Section~3]{gainutdinovtipunin} and also \cite[Remark~2.14]{gainutdinovrunkel}. For a lot of constructions, this holds by definition. For the construction of the \emph{differential graded} modular functor via the homotopy coherent Lego Teichmüller game from \cite{dmf} it requires 
a  proof, especially because the
`Radford map' from Definition~\ref{defradfordmap} was named without any comparison to other definitions.

\begin{proof}[{\slshape Proof of Proposition~\ref{propSmatrix}.}]
	From the construction of \cite[Theorem~5.4]{dmf} and its proof, it follows that $S$ acts on $\cat{C}(I,\mathbb{F}_\bullet)$
	through the automorphism $\mathbb{S}_\bullet:\cat{C}(I,\mathbb{F}_\bullet)\to \cat{C}(I,\mathbb{F}_\bullet)$ induced by the automorphism $\mathbb{S}:\mathbb{F}\to\mathbb{F}$ from \cite[Definition~6.3]{lyuba}. It remains to prove $\mathbb{S}=\Psi \circ \drinfeld$: The automorphism $\mathbb{S}:\mathbb{F}\to\mathbb{F}$
	\cite[Definition~6.3]{lyuba}
	in the description of \cite[Eq.~(2.16)]{fuchsschweigertstigner}
	is given by
	$
	\mathbb{S} := (\varepsilon\otimes \mathbb{F}) \circ \cat{O} \circ (\mathbb{F} \otimes \Lambda)$, 
	where the   map $\cat{O}:\mathbb{F}\otimes\mathbb{F}\to\mathbb{F}\otimes\mathbb{F}$ is induced by the double braiding, more precisely by the maps
	\begin{align} X^\vee \otimes X \otimes Y^\vee  \otimes Y \ra{X^\vee \otimes (c_{Y^\vee,X}\circ c_{X,Y^\vee}) \otimes Y}  X^\vee \otimes X \otimes Y^\vee  \otimes Y \ , \end{align} $\varepsilon:\mathbb{F}\to I$ is the counit and $\Lambda:I\to\mathbb{F}$ is the two-sided integral of $\mathbb{F}$ as Hopf algebra \cite[Theorem~6.9]{shimizuunimodular}. 
	By the universal property of the coend $\mathbb{F}$
	the equality
	\begin{equation}\begin{tikzpicture}
	\begin{pgfonlayer}{nodelayer}
		\node [style=none] (10) at (3, -6.75) {\mlabel{X}};
		\node [style=none] (11) at (2, -6.75) {\mlabel{X^\vee}};
		\node [style=none] (12) at (6, -6.75) {\mlabel{Y^\vee}};
		\node [style=none] (13) at (9, -6.75) {\mlabel{Y}};
		\node [style=none] (14) at (3, -5) {};
		\node [style=none] (15) at (4, -5) {};
		\node [style=none] (16) at (3, -4) {};
		\node [style=none] (17) at (4, -4) {};
		\node [style=none] (18) at (3, -3) {};
		\node [style=none] (19) at (4, -3) {};
		\node [style=none] (20) at (2, -3) {};
		\node [style=none] (21) at (2, -5) {};
		\node [style=none] (22) at (5, -5) {};
		\node [style=none] (23) at (4, -3) {};
		\node [style=none] (24) at (5, -3) {};
		\node [style=none] (25) at (2, -6) {};
		\node [style=none] (26) at (3, -6) {};
		\node [style=none] (27) at (6, -6) {};
		\node [style=none] (28) at (9, -6) {};
		\node [style=none] (29) at (7, -5) {};
		\node [style=none] (30) at (8, -5) {};
		\node [style=none] (31) at (6, -3) {};
		\node [style=none] (32) at (7, -3) {};
		\node [style=none] (33) at (8, -2) {};
		\node [style=none] (34) at (9, -2) {};
		\node [style=none] (35) at (13, -6.75) {\mlabel{X}};
		\node [style=none] (36) at (12, -6.75) {\mlabel{X^\vee}};
		\node [style=none] (37) at (14, -6.75) {\mlabel{Y^\vee}};
		\node [style=none] (38) at (15, -6.75) {\mlabel{Y}};
		\node [style=none] (39) at (13, -5) {};
		\node [style=none] (40) at (14, -5) {};
		\node [style=none] (41) at (13, -4) {};
		\node [style=none] (42) at (14, -4) {};
		\node [style=none] (43) at (13, -3) {};
		\node [style=none] (44) at (14, -3) {};
		\node [style=none] (45) at (12, -3) {};
		\node [style=none] (46) at (12, -5) {};
		\node [style=none] (48) at (14, -2) {};
		\node [style=none] (50) at (12, -6) {};
		\node [style=none] (51) at (13, -6) {};
		\node [style=none] (52) at (14, -6) {};
		\node [style=none] (53) at (15, -6) {};
		\node [style=none] (54) at (15, -2) {};
		\node [style=none] (55) at (10.5, -4) {\mlabel{=}};
	\end{pgfonlayer}
	\begin{pgfonlayer}{edgelayer}
		\draw [in=-90, out=90, looseness=0.75] (15.center) to (16.center);
		\draw [in=-90, out=90, looseness=0.75] (17.center) to (18.center);
		\draw [style=over, in=-90, out=90, looseness=0.75] (14.center) to (17.center);
		\draw [style=over, in=-90, out=90, looseness=0.75] (16.center) to (19.center);
		\draw [bend right=90, looseness=1.50] (18.center) to (20.center);
		\draw (20.center) to (21.center);
		\draw [bend right=90, looseness=1.50] (15.center) to (22.center);
		\draw [style] (21.center) to (25.center);
		\draw [style] (14.center) to (26.center);
		\draw [style] (23.center) to (19.center);
		\draw [style] (24.center) to (22.center);
		\draw (27.center) to (31.center);
		\draw [bend left=90, looseness=1.25] (24.center) to (31.center);
		\draw (29.center) to (32.center);
		\draw [in=90, out=90, looseness=1.25] (32.center) to (23.center);
		\draw [bend right=90, looseness=1.75] (29.center) to (30.center);
		\draw (30.center) to (33.center);
		\draw (28.center) to (34.center);
		\draw [in=-90, out=90, looseness=0.75] (40.center) to (41.center);
		\draw [in=-90, out=90, looseness=0.75] (42.center) to (43.center);
		\draw [style=over, in=-90, out=90, looseness=0.75] (39.center) to (42.center);
		\draw [style=over, in=-90, out=90, looseness=0.75] (41.center) to (44.center);
		\draw [bend right=90, looseness=1.50] (43.center) to (45.center);
		\draw (45.center) to (46.center);
		\draw [style] (46.center) to (50.center);
		\draw [style] (39.center) to (51.center);
		\draw [style] (48.center) to (44.center);
		\draw (40.center) to (52.center);
		\draw (54.center) to (53.center);
	\end{pgfonlayer}
\end{tikzpicture}\qquad \text{implies}\qquad 
	\begin{tikzpicture}[scale=0.65]
	\begin{pgfonlayer}{nodelayer}
		\node [style=none] (11) at (4, -6.25) {\mlabel{\mathbb{F}}};
		\node [style=none] (12) at (6.5, -0.75) {\mlabel{\mathbb{F}}};
		\node [style=none] (15) at (5, -3.25) {};
		\node [style=none] (18) at (5, -2.5) {};
		\node [style=none] (20) at (4, -2.5) {};
		\node [style=none] (21) at (4, -5) {};
		\node [style=none] (22) at (6.5, -3.25) {};
		\node [style=none] (24) at (6.5, -1.5) {};
		\node [style=none] (25) at (4, -5.5) {};
		\node [style=none] (27) at (5.75, -4.25) {};
		\node [style=flat box 2] (31) at (4, -3.5) {\mlabel{\drinfeld}};
		\node [style=none] (32) at (5.75, -5.5) {};
		\node [style=none] (33) at (5.75, -6.25) {\mlabel{\mathbb{F}}};
	\end{pgfonlayer}
	\begin{pgfonlayer}{edgelayer}
		\draw [bend right=90, looseness=1.50] (18.center) to (20.center);
		\draw (20.center) to (21.center);
		\draw [bend right=90, looseness=2.25] (15.center) to (22.center);
		\draw [style] (21.center) to (25.center);
		\draw [style] (24.center) to (22.center);
		\draw [style, in=90, out=-90] (18.center) to (15.center);
		\draw (32.center) to (27.center);
	\end{pgfonlayer}
\end{tikzpicture} \ \stackrel{(*)}{=} \ \begin{tikzpicture}[scale=0.65]
	\begin{pgfonlayer}{nodelayer}
		\node [style=none] (0) at (12, -5) {};
		\node [style=none] (1) at (14, -5) {};
		\node [style=none] (2) at (14, -1) {};
		\node [style=flat box 2] (3) at (12, -1.5) {\mlabel{\varepsilon}};
		\node [style=flat box] (4) at (13, -3.5) {\mlabel{\mathcal{O}}};
		\node [style=none] (5) at (14, -0.5) {\mlabel{\mathbb{F}}};
		\node [style=none] (6) at (12, -5.75) {\mlabel{\mathbb{F}}};
		\node [style=none] (7) at (14, -5.75) {\mlabel{\mathbb{F}}};
	\end{pgfonlayer}
	\begin{pgfonlayer}{edgelayer}
		\draw (0.center) to (3);
		\draw (1.center) to (2.center);
	\end{pgfonlayer}
\end{tikzpicture} \ . 
	\label{proofS2eqn2}
	\end{equation}
	Now we precompose with the integral $\Lambda$ in the respective right slot on the left and right hand side of~$(*)$.
	On the left hand side, we are then left with $\Psi \circ \drinfeld$ 
	thanks to
	\begin{equation}	  \begin{tikzpicture}
	\begin{pgfonlayer}{nodelayer}
		\node [style=none] (0) at (5.5, -5.25) {\mlabel{\mathbb{A}}};
		\node [style=none] (1) at (8, -1.25) {\mlabel{\mathbb{F}}};
		\node [style=none] (2) at (6.5, -2.25) {};
		\node [style=none] (3) at (6.5, -2.25) {};
		\node [style=none] (4) at (5.5, -2.25) {};
		\node [style=none] (5) at (5.5, -3.75) {};
		\node [style=none] (6) at (8, -2.25) {};
		\node [style=none] (7) at (8, -1.75) {};
		\node [style=none] (8) at (5.5, -4.75) {};
		\node [style=flat box 2] (9) at (7.25, -4.25) {\mlabel{\Lambda}};
		\node [style=none] (10) at (7.25, -3) {};
	\end{pgfonlayer}
	\begin{pgfonlayer}{edgelayer}
		\draw [bend right=90, looseness=1.50] (3.center) to (4.center);
		\draw (4.center) to (5.center);
		\draw [bend right=90, looseness=1.75] (2.center) to (6.center);
		\draw [style] (5.center) to (8.center);
		\draw [style] (7.center) to (6.center);
		\draw [style] (9) to (10.center);
		\draw [style, in=90, out=-90] (3.center) to (2.center);
	\end{pgfonlayer}
\end{tikzpicture}   \ = \  \Psi    \ ,  
	\end{equation}
	see~\cite[Section~7]{E2}.
	On the right hand side, we find $\mathbb{S}$.
	This shows $\mathbb{S}=\Psi\circ\drinfeld$ and concludes the proof.
\end{proof}

\begin{example}\label{exhopfalgebra}
	Let $A$ be a ribbon factorizable Hopf algebra. 
	Then the category of finite-dimensional $A$-modules is modular (see~\cite{ntv} for the semisimple case and e.g.~\cite[Section~2.3]{svea2} for the non-semisimple case).
	By \cite[Theorem~7.4.13]{kl} the Lyubashenko coend $\mathbb{F}$ is isomorphic to $A^*_\text{coadj}$, the dual of $A$ with coadjoint $A$-action
	$
	A \otimes A^* \to A^* $ sending $a \otimes \alpha \in A \otimes A^* $ to the linear form on $A$ that sends $b\in A$ to $\alpha\left(      S(a'ba''    )\right)$, 
	where $\Delta a = a'\otimes a''$ is the Sweedler notation for the coproduct and $S$ is the antipode. 
	The Hochschild complex of $A$ is  equivalent to $\Hom_A(k,   {  A^*_\text{coadj}  } _\bullet  )$,
	where $ {  A^*_\text{coadj}  } _\bullet\to    A^*_\text{coadj}  $ is a projective resolution.
	If $A$ is the Drinfeld double $D(G)$ of a finite group $G$ (the category will be non-semisimple if the characteristic of $k$ divides $|G|$), 
	the complex $\Hom_A(k,   {  A^*_\text{coadj}  } _\bullet  )$ is equivalent to $C_*(\PBun_G(\mathbb{T}^2);k)$, the $k$-chains on the groupoid $\PBun_G(\mathbb{T}^2)$ of $G$-bundles over the torus \cite[Lemma~3.2]{dva}, and the mapping class group action is the obvious geometric one. For an arbitrary ribbon factorizable Hopf algebra $A$, the mapping class group action comes from an action of the braid group $B_3$ on three strands on $A^*_\text{coadj}$ \cite{svea}, i.e.\ it descends along the epimorphism $B_3 \to \SL(2,\mathbb{Z})$. This remains even true for arbitrary modular categories \cite{dva}.
\end{example}

\section{The differential graded Verlinde algebra}
In order to formulate the differential graded Verlinde formula,
the differential graded Verlinde algebra needs to be constructed as an $E_2$-algebra.
For the Hochschild chains (where the $E_2$-structure will turn out to be less complicated than on the cochains), this is already 
accomplished
in previous work by applying the functoriality of the Hochschild chains to the braided monoidal product:

\begin{proposition}[$\text{\cite[Proposition~3.11]{dva}}$]\label{propdgva}
	The Hochschild chain complex of a braided finite tensor category carries a non-unital $E_2$-structure
	induced by the braided monoidal product.
\end{proposition}	

For the Hochschild cochains, 
the situation is more difficult, and we will use the following  construction principle for $E_2$-algebras from 
homotopy invariants of a braided commutative algebra in a braided finite tensor category
that we have developed in \cite{E2} using the homotopy theory of braided operads.
Recall that a \emph{braided commutative algebra} $\somealg$ in a braided finite tensor category $\cat{B}$ is an algebra with multiplication $\mu:\somealg\otimes\somealg\to\somealg$ satisfying $\mu \circ c_{\somealg,\somealg}=\mu$.

\begin{theorem}[$\text{\cite[Theorem~3.6]{E2}}$]\label{thmE2derivedhom}
	Let $\mathbb{T} \in \cat{C}$ be an algebra in a finite tensor category $\cat{C}$
	together with a lift to a braided commutative algebra $\somealg \in Z(\cat{C})$ in the Drinfeld center.
	Then the multiplication of $\mathbb{T}$ 
	and the half braiding of $\mathbb{T}$
	induce the structure of an $E_2$-algebra on the cochain complex $\cat{C}(I,\mathbb{T}^\bullet)$ of homotopy invariants of $\mathbb{T}$ (here $\mathbb{T}^\bullet$ is an injective resolution of $\mathbb{T}$).
	The construction is natural in $\mathbb{T}$.
	If $\cat{C}$ is unimodular and pivotal and $\somealg$ has trivial balancing in $Z(\cat{C})$, the $E_2$-algebra structure on $\cat{C}(I,\mathbb{T}^\bullet)$ canonically extends to a framed $E_2$-algebra structure.
\end{theorem}

In order to apply this result, we need a braided commutative algebra. To this end, recall the following result: 

\begin{proposition}[$\text{Lyubashenko \cite[Section~2]{lyuba}}$]\label{propoFmult}
	Let $\cat{C}$ be a  braided finite tensor category.
	Then the maps
	$
	X^\vee \otimes X \otimes Y^\vee \otimes Y \ra{  X^\vee \otimes c_{X,Y^\vee\otimes Y}   } (Y\otimes X)^\vee \otimes Y\otimes X
	$
	induce a 	map $\mu : \mathbb{F}\otimes\mathbb{F}\to\mathbb{F}$ that endows the canonical coend $\mathbb{F}$ with the structure of a unital associative algebra in $\cat{C}$.
\end{proposition}

The maps
\begin{align}\begin{tikzpicture}
	\begin{pgfonlayer}{nodelayer}
		\node [style=none] (0) at (-8, -1) {};
		\node [style=none] (1) at (-7, -1) {};
		\node [style=none] (2) at (-8, 0) {};
		\node [style=none] (3) at (-7, 0) {};
		\node [style=none] (4) at (-9, 1) {};
		\node [style=none] (5) at (-9, 0) {};
		\node [style=none] (6) at (-9, -1) {};
		\node [style=none] (7) at (-8, 1) {};
		\node [style=none] (8) at (-7, 1) {};
		\node [style=none] (9) at (-9, -1.5) {\mlabel{X^\vee}};
		\node [style=none] (10) at (-8, -1.5) {\mlabel{X}};
		\node [style=none] (11) at (-7, -1.5) {\mlabel{Y}};
	\end{pgfonlayer}
	\begin{pgfonlayer}{edgelayer}
		\draw (1.center) to (2.center);
		\draw [style=over, in=-90, out=90] (0.center) to (3.center);
		\draw (6.center) to (5.center);
		\draw [in=270, out=90, looseness=0.75] (5.center) to (7.center);
		\draw [style=over] (2.center) to (4.center);
		\draw (3.center) to (8.center);
	\end{pgfonlayer}
\end{tikzpicture} \ : \ X^\vee \otimes X \otimes Y \to Y\otimes X^\vee \otimes X\label{eqndolphin}
\end{align}
induce a half braiding for $\mathbb{F}$ that is often referred to as \emph{dolphin half braiding} and in other contexts as \emph{field goal transform}. 
We denote this lift of $\mathbb{F}$ to the Drinfeld center by
$\mathbb{F}\dolph$. 
It is proven in \cite[Definition~4 \& Proposition~5]{neuchlschauenburg}
that the multiplication
$\mu : \mathbb{F}\otimes\mathbb{F}\to\mathbb{F}$ from Proposition~\ref{propoFmult} lifts to a braided commutative multiplication on 
$\mathbb{F}\dolph\in Z(\cat{C})$, see also 
\cite[Lemma~2.8]{cardycase}. Therefore, we obtain an $E_2$-structure on the homotopy invariants $\cat{C}(I,\mathbb{F}^\bullet)$ by Theorem~\ref{thmE2derivedhom}, see also \cite[Example~3.8]{E2}.

\begin{definition}\label{defdolphinalgebra}
	For any braided finite tensor category $\cat{C}$, 
	we denote $\cat{C}(I,\mathbb{F}^\bullet)$ with its $E_2$-structure coming from 
	the multiplication $\mu : \mathbb{F}\otimes\mathbb{F}\to\mathbb{F}$ from Proposition~\ref{propoFmult} 
	and the dolphin half braiding of $\mathbb{F}$ 
	by $\mathfrak{A}_\cat{C}\dolph$ and refer to this $E_2$-algebra as the \emph{dolphin algebra} of $\cat{C}$.
\end{definition}

\begin{remark}\label{remrelclassfunction}
	Thanks to $\cat{C}(I,\mathbb{F})\cong Z(\cat{C})(\coalg,\coalg)$,
	the vector space $\cat{C}(I,\mathbb{F})$ becomes an algebra, the \emph{algebra $\catf{CF}(\cat{C})$ of class functions of $\cat{C}$}. 
	But the vector space $\cat{C}(I,\mathbb{F})$ also coincides with $H^0(\mathfrak{A}_\cat{C}\dolph)$, and we can extract from 
	\cite[Proposition~3.13]{shimizucf} that $H^0(\mathfrak{A}_\cat{C}\dolph)= \catf{CF}(\cat{C})$ as algebras.
\end{remark}

With these preparations,
obtaining the differential graded Verlinde algebra on the Hochschild 
\emph{cochain} complex is relatively straightforward: 

\begin{theorem}[Verlinde algebra on the Hochschild cochain complex]\label{thmverlindecochain}
	Let $\cat{C}$ be a unimodular braided finite tensor category with a fixed trivialization $D\cong I$ of the distinguished invertible object.
	Then the Hochschild cochain complex $\rint_{X\in \Proj \cat{C}} \cat{C}(X,X)$ inherits from its braided monoidal product the structure of an $E_2$-algebra $\bspace{\rint_{X\in \Proj \cat{C}} \cat{C}(X,X)}{\otimes}$ whose multiplication we denote again by $\otimes$. 
\end{theorem}
We refer to this $E_2$-algebra as the \emph{differential graded Verlinde algebra on the Hochschild cochain complex of $\cat{C}$}.

\begin{proof}
	By Proposition~\ref{propcocomplex},
	we obtain an equivalence of chain complexes
	\begin{align}
	\rint_{X\in \Proj \cat{C}} \cat{C}(X,X)\simeq \cat{C}(I,\mathbb{A}^\bullet)\ra{\Psi^\bullet} \cat{C}(I,\mathbb{F}^\bullet) \ ,    \label{eqnpsidolphin}
	\end{align}
	which, with suitable models for the homotopy end and a resolution of $\mathbb{A}$,  is actually an isomorphism.
	As a consequence, there is, up to homotopy, a unique $E_2$-algebra $\bspace{\rint_{X\in \Proj \cat{C}} \cat{C}(X,X)}{\otimes}$
	such that
	$
	\bspace{\rint_{X\in \Proj \cat{C}} \cat{C}(X,X)}{\otimes}    \ra{\eqref{eqnpsidolphin}} \mathfrak{A}_\cat{C}\dolph\label{eqndolphinequiv}
	$
	is an equivalence (or depending on the model even an isomorphism) of $E_2$-algebras.
\end{proof}

We will later need a technical result on the dolphin algebra from Definition~\ref{defdolphinalgebra}.
This will rely on our construction of $E_2$-algebras (Theorem~\ref{thmE2derivedhom}), but also on the following  well-known result:

\begin{proposition}
	\label{propdrinfeldmapmorphismus}
	The Drinfeld map $\drinfeld:\mathbb{F}\to\mathbb{A}$ of a finite braided tensor category is a map of algebras $(\mathbb{F},\mu)\to(\mathbb{A},\gamma)$. Moreover, it lifts to a morphism $\mathbb{F}\dolph\to \alg$ of algebras in $Z(\cat{C})$, where $\mathbb{F}\dolph\in Z(\cat{C})$ is the canonical coend of $\cat{C}$ equipped with the dolphin half braiding and $\alg$ is the canonical algebra of $Z(\cat{C})$. 
\end{proposition}

This statement, at least in the Hopf algebraic case, goes back to Drinfeld \cite{drinfeld}. 
If $\cat{C}$ is modular, then Proposition~\ref{propdrinfeldmapmorphismus} tells us that $\mathbb{F}$ and $\mathbb{A}$ are isomorphic as algebras. In this situation, a related statement taking also comultiplications into account is  given in  \cite[Theorem~5.16]{eilind}. For us, however, the version given in Proposition~\ref{propdrinfeldmapmorphismus} is sufficient.
Since Proposition~\ref{propdrinfeldmapmorphismus} is quite vital and since the argument behind it is very insightful, we give a short graphical proof.

\begin{proof}[{\slshape Proof of Proposition~\ref{propdrinfeldmapmorphismus}}]
	In order to prove $\drinfeld \circ \mu = \gamma \circ (\drinfeld \otimes\drinfeld)$, it suffices
	by the universal property of (co)ends
	to prove that for $X,Y,Z\in\cat{C}$ the restriction of the $Z$-component 
	to $X^\vee \otimes X\otimes Y^\vee \otimes Y$ of both maps agree. 
	We denote this component by $(\drinfeld \circ \mu)_{X,Y}^Z$ and $(\gamma \circ (\drinfeld \otimes\drinfeld))_{X,Y}^Z$, respectively. 
	Now the proof follows from the following computation in the graphical calculus:
	\begin{equation}\begin{tikzpicture}
	\begin{pgfonlayer}{nodelayer}
		\node [style=none] (14) at (0, 5) {};
		\node [style=none] (15) at (1, 5) {};
		\node [style=none] (16) at (0, 6) {};
		\node [style=none] (17) at (1, 6) {};
		\node [style=none] (18) at (0, 7) {};
		\node [style=none] (19) at (1, 7) {};
		\node [style=none] (20) at (-1, 7) {};
		\node [style=none] (21) at (-1, 5) {};
		\node [style=none] (22) at (1.5, 5) {};
		\node [style=none] (23) at (1, 9) {};
		\node [style=none] (24) at (1.5, 7.5) {};
		\node [style=none] (25) at (-1, 4) {};
		\node [style=none] (26) at (0, 4) {};
		\node [style=none] (27) at (-12, 5.25) {};
		\node [style=none] (29) at (-11, 6) {};
		\node [style=none] (30) at (-10.5, 6) {};
		\node [style=none] (31) at (-12.25, 6) {};
		\node [style=none] (32) at (-10.75, 4.75) {};
		\node [style=none] (33) at (-10.5, 7.25) {};
		\node [style=none] (34) at (-12, 7.25) {};
		\node [style=none] (35) at (-11.25, 7.25) {};
		\node [style=none] (36) at (-9.5, 4.75) {};
		\node [style=none] (37) at (-12, 4) {};
		\node [style=none] (38) at (-13, 4) {};
		\node [style=none] (39) at (-14, 4) {};
		\node [style=none] (40) at (-15, 4) {};
		\node [style=none] (41) at (-13, 7.25) {};
		\node [style=none] (42) at (-15, 7.25) {};
		\node [style=none] (43) at (-10.5, 8.75) {};
		\node [style=none] (44) at (-9.5, 8.75) {};
		\node [style=none] (67) at (-3.75, 5.75) {};
		\node [style=none] (68) at (-4.5, 6.5) {};
		\node [style=none] (69) at (-3.5, 7.75) {};
		\node [style=none] (70) at (-4.5, 7.75) {};
		\node [style=none] (71) at (-4.5, 5.5) {};
		\node [style=none] (72) at (-4.5, 8.75) {};
		\node [style=none] (75) at (-3, 5.5) {};
		\node [style=none] (76) at (-4.5, 4) {};
		\node [style=none] (77) at (-5.5, 4) {};
		\node [style=none] (78) at (-6.5, 4) {};
		\node [style=none] (79) at (-7.5, 4) {};
		\node [style=none] (80) at (-5.5, 5) {};
		\node [style=none] (81) at (-7.5, 6) {};
		\node [style=none] (82) at (-4.5, 8.75) {};
		\node [style=none] (83) at (-3, 8.75) {};
		\node [style=none] (84) at (-6.5, 6) {};
		\node [style=none] (85) at (3, 5) {};
		\node [style=none] (86) at (4, 5) {};
		\node [style=none] (87) at (3, 6) {};
		\node [style=none] (88) at (4, 6) {};
		\node [style=none] (89) at (3, 7) {};
		\node [style=none] (90) at (4, 7) {};
		\node [style=none] (91) at (2, 7) {};
		\node [style=none] (92) at (2, 5) {};
		\node [style=none] (93) at (5, 5) {};
		\node [style=none] (94) at (4, 7.5) {};
		\node [style=none] (95) at (5, 9) {};
		\node [style=none] (96) at (2, 4) {};
		\node [style=none] (97) at (3, 4) {};
		\node [style=none] (98) at (-15, 3.25) {\mlabel{X^\vee}};
		\node [style=none] (99) at (-14, 3.25) {\mlabel{X}};
		\node [style=none] (100) at (-13, 3.25) {\mlabel{Y^\vee}};
		\node [style=none] (101) at (-12, 3.25) {\mlabel{Y}};
		\node [style=none] (102) at (-1, 3.25) {\mlabel{X^\vee}};
		\node [style=none] (103) at (0, 3.25) {\mlabel{X}};
		\node [style=none] (104) at (2, 3.25) {\mlabel{Y^\vee}};
		\node [style=none] (105) at (3, 3.25) {\mlabel{Y}};
		\node [style=none] (106) at (-7.5, 3.25) {\mlabel{X^\vee}};
		\node [style=none] (107) at (-6.5, 3.25) {\mlabel{X}};
		\node [style=none] (108) at (-5.5, 3.25) {\mlabel{Y^\vee}};
		\node [style=none] (109) at (-4.5, 3.25) {\mlabel{Y}};
		\node [style=none] (110) at (-10.5, 9.25) {\mlabel{Z}};
		\node [style=none] (111) at (-9.5, 9.25) {\mlabel{Z^\vee}};
		\node [style=none] (112) at (-4.5, 9.25) {\mlabel{Z}};
		\node [style=none] (113) at (-3, 9.25) {\mlabel{Z^\vee}};
		\node [style=none] (115) at (1, 9.5) {\mlabel{Z}};
		\node [style=none] (116) at (5, 9.5) {\mlabel{Z^\vee}};
		\node [style=none] (117) at (-8.5, 6) {\mlabel{=}};
		\node [style=none] (118) at (-2, 6) {\mlabel{=}};
		\node [style=none] (120) at (-17.5, 6) {\mlabel{(\drinfeld \circ \mu)_{X,Y}^Z =}};
		\node [style=none] (122) at (8.25, 6) {\mlabel{=(\gamma \circ (\drinfeld \otimes \drinfeld))_{X,Y}^Z}};
	\end{pgfonlayer}
	\begin{pgfonlayer}{edgelayer}
		\draw [in=-90, out=90, looseness=0.75] (15.center) to (16.center);
		\draw [in=-90, out=90, looseness=0.75] (17.center) to (18.center);
		\draw [style=over, in=-90, out=90, looseness=0.75] (14.center) to (17.center);
		\draw [style=over, in=-90, out=90, looseness=0.75] (16.center) to (19.center);
		\draw [bend right=90, looseness=1.50] (18.center) to (20.center);
		\draw (20.center) to (21.center);
		\draw [bend right=90, looseness=1.50] (15.center) to (22.center);
		\draw [style] (21.center) to (25.center);
		\draw [style] (14.center) to (26.center);
		\draw [style] (23.center) to (19.center);
		\draw [style] (24.center) to (22.center);
		\draw [in=-90, out=90, looseness=0.75] (32.center) to (31.center);
		\draw [in=-90, out=90, looseness=0.50] (29.center) to (34.center);
		\draw [in=-90, out=90, looseness=0.50] (30.center) to (35.center);
		\draw [style=over, in=-90, out=90, looseness=0.50] (27.center) to (29.center);
		\draw [style=over, in=-90, out=90, looseness=0.75] (31.center) to (33.center);
		\draw [bend right=90, looseness=1.25] (32.center) to (36.center);
		\draw (27.center) to (37.center);
		\draw (41.center) to (38.center);
		\draw [bend left=90, looseness=1.75] (41.center) to (34.center);
		\draw [style=over, in=-90, out=90, looseness=0.75] (39.center) to (30.center);
		\draw [bend right=90] (35.center) to (42.center);
		\draw (42.center) to (40.center);
		\draw (33.center) to (43.center);
		\draw (36.center) to (44.center);
		\draw [bend right=90] (71.center) to (75.center);
		\draw (80.center) to (77.center);
		\draw (81.center) to (79.center);
		\draw (72.center) to (82.center);
		\draw (75.center) to (83.center);
		\draw [style=over, in=-90, out=90, looseness=0.50] (76.center) to (67.center);
		\draw [in=90, out=90, looseness=0.50] (67.center) to (80.center);
		\draw [style=over] (71.center) to (68.center);
		\draw (68.center) to (70.center);
		\draw (78.center) to (84.center);
		\draw [style=over, in=-90, out=90, looseness=0.50] (84.center) to (69.center);
		\draw [in=-90, out=90, looseness=0.75] (86.center) to (87.center);
		\draw [in=-90, out=90, looseness=0.75] (88.center) to (89.center);
		\draw [style=over, in=-90, out=90, looseness=0.75] (85.center) to (88.center);
		\draw [style=over, in=-90, out=90, looseness=0.75] (87.center) to (90.center);
		\draw [bend right=90, looseness=1.50] (89.center) to (91.center);
		\draw (91.center) to (92.center);
		\draw [bend right=90, looseness=1.50] (86.center) to (93.center);
		\draw [style] (92.center) to (96.center);
		\draw [style] (85.center) to (97.center);
		\draw [style] (94.center) to (90.center);
		\draw [style] (95.center) to (93.center);
		\draw [bend left=90, looseness=0.75] (24.center) to (94.center);
		\draw [in=135, out=90, looseness=0.75] (81.center) to (69.center);
		\draw [style=over] (70.center) to (72.center);
	\end{pgfonlayer}
\end{tikzpicture}
	\end{equation}\normalsize
	A similar computation proves that $\drinfeld$ is also a morphism in the Drinfeld center $Z(\cat{C})$.
\end{proof}

\begin{proposition}\label{propdolphdrinfeld}
	For any braided finite tensor category $\cat{C}$, the Drinfeld map $\drinfeld :\mathbb{F}\to\mathbb{A}$ induces a map of $E_2$-algebras
	$\drinfeld^\bullet: \mathfrak{A}_\cat{C}\dolph      \to             \cat{C}(I,\mathbb{A}^\bullet)$
	from the dolphin algebra to the homotopy invariants of the canonical end.
	This map is an equivalence (or isomorphism, for suitable models) if and only if the braiding of $\cat{C}$ is non-degenerate.
\end{proposition}

\begin{proof}
	It follows from Proposition~\ref{propdrinfeldmapmorphismus}	and the naturality statement contained in Theorem~\ref{thmE2derivedhom}
	that $\drinfeld^\bullet$ is a map of $E_2$-algebras.
	If $\cat{C}$ is non-degenerate, then $\drinfeld^\bullet$ is an equivalence because the Drinfeld map is an isomorphism \cite{shimizumodular}. If conversely $\drinfeld^\bullet$ is an equivalence, then, in particular, the map $\cat{C}(I,\mathbb{F})\to\cat{C}(I,\mathbb{A})$ induced by $\drinfeld$ in zeroth cohomology is an isomorphism. By applying again Shimizu's results \cite{shimizumodular} this suffices to ensure non-degeneracy.
\end{proof}

\section{The trace field theory and the block diagonal product on Hochschild chains}
As a final preparation for the formulation and proof of the differential graded Verlinde formula, we need to make contact to modified traces by means of topological conformal field theory:
Following \cite{fss},
the \emph{(right) Nakayama functor} $\naka :\cat{C}\to\cat{C}$
of a finite category $\cat{C}$
can be described in a Morita-invariant way
by
\begin{align} \naka X:= \int^{Y\in \cat{C}} \cat{C}(X,Y)^* \otimes Y \quad \text{for}\quad X\in\cat{C} \ .   \label{defeqnnaka}
\end{align}  
For a finite tensor category $\cat{C}$, there is  by \cite[Theorem~3.18]{fss}
a natural isomorphism 
$
\naka \cong D^{-1}\otimes -^{\vee\vee} 
$
turning $\naka$ into an equivalence from $\cat{C}$ as regular right $\cat{C}$-module category to $\cat{C}$ as regular $^{\vee\vee}$-twisted right $\cat{C}$-module category.
A trivialization of $\naka$ as right $\cat{C}$-module functor relative to a pivotal structure on $\cat{C}$
is referred to as \emph{symmetric Frobenius structure} in \cite{trace}; it
amounts to a pivotal structure on $\cat{C}$
and a trivialization $D\cong I$ of the distinguished invertible object.

\begin{theorem}[$\text{\cite[Theorem~3.6]{trace}, see also \cite{ss21}}$]\label{thmmtrace}
	For any finite tensor category $\cat{C}$ with symmetric Frobenius structure, the tensor ideal $\Proj \cat{C}$ canonically comes with a Calabi-Yau structure.
	The associated trace functions form a right modified trace on $\Proj \cat{C}$.
\end{theorem}

A modified trace on $\Proj \cat{C}$
is a cyclic, non-degenerate trace satisfying the partial trace property, see \cite{geerpmturaev,mtrace1,mtrace2,mtrace3,mtrace} for details. Under the above assumptions, it is unique up to invertible scalar. Note that in Theorem~\ref{thmmtrace} \emph{one specific} modified trace is obtained through the trivialization of the Nakayama functor; no other choice is made.
If $\cat{C}$ is a finite tensor category with symmetric Frobenius structure, 
the \emph{trace field theory} of $\cat{C}$ \cite{trace} is defined as the open-closed topological conformal field theory $\Phi_\cat{C}:\OC\to\Ch$
that the Calabi-Yau structure on $\Proj \cat{C}$ coming from the fixed trivialization of $\naka$ gives rise to by results of Costello, Egas Santander, Wahl and Westerland \cite{costellotcft,egas,wahlwesterland}. 
Here $\OC$ is a differential graded version of the open-closed two-dimensional bordism category with the projective objects of $\cat{C}$ as label set aka set of `D-branes', see \cite{costellotcft} for the definition and \cite[Section~4]{trace} for a very brief review. 
By evaluation of $\Phi_\cat{C}$
on the pair of pants, one obtains the \emph{block diagonal $\star$-product}
of the finite tensor category $\cat{C}$ with symmetric Frobenius structure:
\begin{align}
\star := \Phi_\cat{C}\left(\begin{tikzpicture}[scale=0.4]
	\begin{pgfonlayer}{nodelayer}
		\node [style=none] (0) at (1.75, 9.75) {};
		\node [style=none] (1) at (1.75, 8.75) {};
		\node [style=none] (2) at (1.75, 8.25) {};
		\node [style=none] (3) at (1.75, 7.25) {};
		\node [style=none] (4) at (3.5, 9) {};
		\node [style=none] (5) at (3.5, 8) {};
		\node [style=none] (6) at (2.5, 6.75) {\phantom{x}};
	\end{pgfonlayer}
	\begin{pgfonlayer}{edgelayer}
		\draw [bend left=90] (0.center) to (1.center);
		\draw [bend right=90] (0.center) to (1.center);
		\draw [bend left=90] (2.center) to (3.center);
		\draw [bend right=90] (2.center) to (3.center);
		\draw [bend left=90] (4.center) to (5.center);
		\draw [bend right=90] (4.center) to (5.center);
		\draw [bend right=90, looseness=1.75] (2.center) to (1.center);
		\draw [in=180, out=0, looseness=0.50] (0.center) to (4.center);
		\draw [in=-180, out=0, looseness=0.75] (3.center) to (5.center);
	\end{pgfonlayer}
\end{tikzpicture}\right) : \lint^{X\in\Proj\cat{C}} \cat{C}(X,X) \otimes \lint^{X\in\Proj\cat{C}} \cat{C}(X,X)\to \lint^{X\in\Proj\cat{C}} \cat{C}(X,X) \, . \label{eqnstarproduct}
\end{align}
By construction this is a non-unital $E_2$-multiplication (non-unital because the bordism without incoming boundary that would usually provide the unit is not admissible in $\OC$).
It is the multiplication afforded by the cyclic Deligne Conjecture applied to the Calabi-Yau structure on $\Proj \cat{C}$ coming from the modified trace, where the last connection to the modified trace is a consequence of Theorem~\ref{thmmtrace}.
Wahl and Westerland \cite{wahlwesterland} prove that a product extracted in the way~\eqref{eqnstarproduct}
from the open-closed topological conformal field theory of a symmetric Frobenius algebra (or, more generally, a Calabi-Yau category) is, up to homotopy, supported 
in homological degree zero.
Specifically for finite tensor categories, this product is further investigated in~\cite{trace}. In particular, it is shown in \cite[Proposition~5.3]{trace} that the product $\star$ is block diagonal; this will be spelled out in more detail on page~\pageref{discussblockdiagonal}.
The description~\eqref{eqnstarproduct}
of the block diagonal product $\star$ is entirely topological. We will make  use of this fact later, but we need additionally a description in terms of the canonical coend of our category:
For any finite tensor category $\cat{C}$ (we do not assume a trivialization of $\naka$ for the moment), the maps $ Y \otimes \cat{C}(Y,X)\to X$ for $X,Y\in\cat{C}$ induce maps
\begin{align}
X^\vee \otimes X \to      \left( Y \otimes \cat{C}(Y,X)   \right)^\vee \otimes X \cong Y^\vee \otimes \cat{C}(Y,X)^* \otimes X \to Y^\vee \otimes \naka Y \ , \label{eqnthetaXY}
\end{align}
where we have used the definition of the Nakayama functor in~\eqref{defeqnnaka}. These maps descend to the coend $\mathbb{F}=\int^{X\in\cat{C}} X^\vee \otimes X$ and factor through the end $\int_{Y\in \cat{C}} Y^\vee \otimes \naka Y$; in other words, they yield a map
\begin{align} \mathbb{F}\ra{\cong} \int_{Y\in \cat{C}} Y^\vee \otimes \naka Y\label{eqnisoFnaka}\end{align} which is in fact an isomorphism because it can be obtained by applying the duality functor and the monoidal product to the isomorphism 
\begin{align}   \label{fssisoeqn}  \int^{X \in \cat{C}} X \boxtimes X \cong \int_{X\in\cat{C}} X \boxtimes \naka X \quad \text{in}\quad    \cat{C}^\op \boxtimes \cat{C}
\end{align}
from~\cite[equation~(3.52)]{fss}.

If $\cat{C}$ comes with a symmetric Frobenius structure, we   obtain an isomorphism
\begin{align}
\Omega : \mathbb{F} \ra{\eqref{eqnisoFnaka}} \int_{Y\in \cat{C}} Y^\vee \otimes \naka Y \ra{\naka \cong\id_\cat{C}} \int_{Y\in \cat{C}} Y^\vee \otimes  Y \cong \int_{Y\in \cat{C}} Y\otimes Y^\vee =\mathbb{A} \quad \text{in}\ \cat{C} \ ,     \label{eqndeftheta}
\end{align}
where in the last step we relabel the dummy variable and use the pivotal structure.

\begin{lemma}\label{lemmapsiomega}
	For any finite tensor category $\cat{C}$ with symmetric Frobenius structure,
	the isomorphism $\Omega:\mathbb{F}\to\mathbb{A}$ is the inverse of the Radford map $\Psi : \mathbb{A}\to\mathbb{F}$
	from Definition~\ref{defradfordmap}.
\end{lemma}

\begin{proof}
	Thanks to $\naka \cong D^{-1}\otimes -^{\vee\vee}$,
	we can obtain $\Omega$ by applying the monoidal product to
	\begin{align}
	\int^{X\in\cat{C}} X^\vee \boxtimes X \stackrel{\eqref{fssisoeqn}}{\cong} \int_{X\in\cat{C}} X^\vee \boxtimes D^{-1}\otimes X^{\vee \vee} \cong \int_{X\in\cat{C}} X\boxtimes D^{-1}\otimes X^\vee
	\quad \text{in}\quad \cat{C}\boxtimes\cat{C} \    \label{eqnisofss}
	\end{align} 
	and using afterwards the isomorphism $D\cong I$ that is part of the symmetric Frobenius structure.

	We now need to relate this to
	the Radford map $\Psi:\mathbb{A}\to\mathbb{F}$ that we had defined using the result 
	\cite[Lemma~4.7 \& Theorem~4.10]{shimizuunimodular} on the relation between
	the left adjoint $L:\cat{C}\to Z(\cat{C})$ and the right adjoint $R:\cat{C}\to Z(\cat{C})$ to the forgetful functor $U:Z(\cat{C})\to\cat{C}$.
	In order to understand the relation to $\Omega$,
	denote by $\cat{C}^\text{rev}$ the finite tensor category obtained by reversing the monoidal product of $\cat{C}$. Then $\cat{C}^\text{env}:=\cat{C} \boxtimes \cat{C}^\text{rev}$ is a finite tensor category, and the monoidal product of $\cat{C}$ turns $\cat{C}$ into a $\cat{C}^\text{env}$-module category (the action is $(X\otimes Y).Z:=X\otimes Z\otimes Y$ for $X,Y,Z\in\cat{C}$).
	The comparison of $L$ and $R$ is based on the algebra $B \in \cat{C}^\text{env}$ which is the internal endomorphism object of $I\in\cat{C}$ for the $\cat{C}^\text{env}$-action on $\cat{C}$; explicitly $B=\int^{X\in\cat{C}} X \boxtimes  X^\vee\in \cat{C}^\text{env}$.
	Denote by ${}_B B$ and $B_B$ the object $B\in\cat{C}^\text{env}$ as regular left and right module over itself, respectively. 
	The crucial observation used in \cite{shimizuunimodular} is now that $Z(\cat{C})$ can be identified with the category ${}_B (\cat{C}^\text{env}) _B$  of $B$-bimodules. Under this identification, the left and the right adjoint of $U$ take the form
	\begin{align} L : \cat{C} &\to {}_B (\cat{C}^\text{env}) _B \ , \quad Y \mapsto {}_BB \otimes (Y\boxtimes I) \otimes B_B    \stackrel{\eqref{eqnisofss}}{\cong}     \left( \int_{X\in\cat{C}} X  \boxtimes   D^{-1} \otimes X^\vee \right) \otimes (Y\boxtimes I) \otimes B_B     \ , \\
	R : \cat{C} &\to {}_B (\cat{C}^\text{env}) _B \ , \quad Y \mapsto B_B^\vee \otimes (Y\boxtimes I) \otimes B_B = \left( \int_{X\in \cat{C}} X\boxtimes X^\vee\right) \otimes (Y\boxtimes I)\otimes B_B \ ,
	\end{align}
	thereby leading us to $L\cong R(-\otimes D^{-1})$. In order to obtain $\Psi^{-1}$, we have to postcompose with the forgetful functor $U$, evaluate the resulting isomorphism
	$UL\cong UR(-\otimes D^{-1})$
	at $I$ and use the isomorphism $D\cong I$, but this  is $\Omega$ by the description through~\eqref{eqnisofss}.
\end{proof}

Up to the use of the pivotal structure in the last step in~\eqref{eqndeftheta},
the components of $\Omega$ are the maps
$
\theta_{X,Y} : X^\vee \otimes X \ra{\eqref{eqnthetaXY}}       Y^\vee \otimes \naka Y   \ra{\naka\cong \id_\cat{C}} Y^\vee \otimes  Y   $
which will provide us with maps
\begin{align}
\begin{tikzpicture}
	\begin{pgfonlayer}{nodelayer}
		\node [style=none] (0) at (-3, 1) {};
		\node [style=none] (1) at (-2, 1) {};
		\node [style=none] (2) at (-1, 1) {};
		\node [style=none] (3) at (0, 1) {};
		\node [style=flat box 2] (4) at (-2.5, 2) {\mlabel{\theta_{X,Y}}};
		\node [style=none] (5) at (-3, 3.5) {};
		\node [style=none] (6) at (-2, 2.75) {};
		\node [style=none] (7) at (-3, 0.5) {\mlabel{X^\vee}};
		\node [style=none] (8) at (-2, 0.5) {\mlabel{X}};
		\node [style=none] (9) at (-1, 0.5) {\mlabel{Y^\vee}};
		\node [style=none] (10) at (0, 0.5) {\mlabel{Y}};
		\node [style=none] (11) at (-1, 2.75) {};
		\node [style=none] (12) at (0, 3.5) {};
		\node [style=none] (13) at (-3, 4) {\mlabel{Y^\vee}};
		\node [style=none] (14) at (0, 4) {\mlabel{X}};
	\end{pgfonlayer}
	\begin{pgfonlayer}{edgelayer}
		\draw [bend left=90, looseness=1.75] (6.center) to (11.center);
		\draw (6.center) to (1.center);
		\draw (5.center) to (0.center);
		\draw (11.center) to (2.center);
		\draw (12.center) to (3.center);
	\end{pgfonlayer}
\end{tikzpicture} \  \ :\ \ X^\vee \otimes X \otimes       Y^\vee \otimes Y \to Y^\vee \otimes Y \ . \label{eqnmapsostar}
\end{align}
We have used here that left and right duality coincide thanks to the pivotal structure.

\begin{proposition}\label{propmultostar}
	For any finite tensor category $\cat{C}$ with symmetric Frobenius structure, there is a unique algebra structure 
	$\ostar : \mathbb{F}\otimes\mathbb{F}\to\mathbb{F}$
	characterized by any of the following equivalent descriptions:
	\begin{pnum}
		\item The product $\ostar$ is induced by the maps~\eqref{eqnmapsostar}.
		\item The product $\ostar$ is the unique product on $\mathbb{F}$ turning
		$
		\Psi :   \bspace{\mathbb{A}}{\gamma}      \ra{\cong}\bspace{\mathbb{F}}{\ostar}     
		$
		into	an isomorphism of algebras, where $\Psi : \mathbb{A}\to \mathbb{F}$ is the Radford map.\label{seconddescription}	
	\end{pnum}
\end{proposition}

\begin{proof}
	Thanks to $\Psi^{-1} = \Omega$ by Lemma~\ref{lemmapsiomega},
	it suffices to \emph{define} $\ostar$ via~\eqref{eqnmapsostar} and verify
	$ \gamma \circ (\Omega \otimes \Omega)=\Omega \circ \ostar $.
	This will in particular prove that $\ostar$ actually yields the structure of an algebra on $\mathbb{F}$ (which, just from~\eqref{eqnmapsostar}, would not be clear).
	By definition the components of $\Omega : \mathbb{F}\to\mathbb{A}$ are given by
	$ \widetilde \theta_{X,Y} : X^\vee \otimes X \ra{\theta_{X,Y^\vee}} Y^{\vee \vee} \otimes Y^\vee \ra{\omega} Y\otimes Y^\vee$. 
	We will now describe the components in terms of the Calabi-Yau structure on $\Proj\cat{C}$. For this, we may assume that $X$ and $Y$ are projective, which is justified by  \cite[Proposition~5.1.7]{kl}.
	Now $\widetilde \theta_{X,Y}$ is explicitly given by the composition
	\begin{align}
	X^\vee \otimes X \to      \left( Y^\vee \otimes \cat{C}(Y^\vee,X)   \right)^\vee \otimes X \stackrel{\omega}{\cong} Y \otimes \cat{C}(Y^\vee,X)^* \otimes X \stackrel{(*)}{\cong} Y\otimes \cat{C}(X,Y^\vee)\otimes X \to Y \otimes Y^{\vee} \  , \label{eqncomppsix}
	\end{align} \normalsize
	where in step~$(*)$ we use the isomorphism
	$\cat{C}(Y^\vee,X)^*\cong \cat{C}(X,Y^{\vee})$ afforded by the Calabi-Yau structure (this uses that $X$ and $Y$ are assumed to be projective). 
	The isomorphism
	$\cat{C}(Y^\vee,X)^*\cong \cat{C}(X,Y^{\vee})$ 
	can be expressed through the coproducts and the unit of the 
	Calabi-Yau category $\Proj\cat{C}$.
	Here by coproduct we mean the map
	$
	\Delta_{X,Y^\vee} : \cat{C}(X,X)\to\cat{C}(X,Y^\vee)\otimes\cat{C}(Y^\vee,X)
	$
	obtained by dualizing the composition over $Y^\vee$ via the Calabi-Yau structure.
	This means that the isomorphism $\cat{C}(Y^\vee,X)^*\cong \cat{C}(X,Y^{\vee})$  is given by
	the commuting square
	\begin{equation}
	\begin{tikzcd}
	\ar[rrr,"\text{$\cat{C}(Y^\vee,X)^* \otimes \id_X$}"] \ar[dd,swap,"\cong"] \cat{C}(Y^\vee,X)^*  && & \cat{C}({^\vee Y},X)^*\otimes \cat{C}(X,X)   \ar[dd,"\text{$\cat{C}(Y^\vee,X)^* \otimes \Delta_{X,Y^\vee}$}"]  \\ \\
	\cat{C}(X,Y^\vee)    && &       	\ar[lll,"\text{evaluation}"]      	\cat{C}(Y^\vee,X)^*\otimes \cat{C}(X,Y^\vee)\otimes\cat{C}(Y^\vee,X)    \ .    \\ 
	\end{tikzcd} \label{eqnsquarepsi}
	\end{equation}
	In order to be even more explicit, we  use Sweedler notation $\Delta_{X,Y^\vee}(\id_X)=\alpha(X,Y)' \otimes \alpha(X,Y)'' \in \cat{C}(X,Y^\vee)\otimes\cat{C}(Y^\vee,X)$.
	With this notation,
	\begin{align} \widetilde \theta_{X,Y}  \ = \  	\begin{tikzpicture}
	\begin{pgfonlayer}{nodelayer}
		\node [style=none] (0) at (-6, 2.25) {};
		\node [style=none] (1) at (-2, 2.25) {};
		\node [style=flat box 2] (3) at (-2, 3.5) {{\footnotesize $\alpha(X,Y)'$}};
		\node [style=flat box 2] (4) at (-6, 3.5) {{\footnotesize${\alpha(X,Y)''}^\vee$}};
		\node [style=none] (5) at (-6, 4.75) {};
		\node [style=none] (6) at (-2, 4.75) {};
		\node [style=none] (7) at (-2, 1.75) {\mlabel{X}};
		\node [style=none] (8) at (-6, 1.75) {\mlabel{X^\vee}};
		\node [style=none] (9) at (-6, 5.25) {\mlabel{Y}};
		\node [style=none] (10) at (-2, 5.25) {\mlabel{Y^\vee}};
	\end{pgfonlayer}
	\begin{pgfonlayer}{edgelayer}
		\draw (0.center) to (4);
		\draw (4) to (5.center);
		\draw (3) to (6.center);
		\draw (1.center) to (3);
	\end{pgfonlayer}
\end{tikzpicture} \  \ :\ \  X^\vee \otimes X&\to Y\otimes Y^\vee \ , \label{eqnpsiinvy}
	\end{align}
	where, by slight abuse of notation, we see ${\alpha(X,Y)''}^\vee$ as a map $X^\vee \to Y$ via the pivotal structure.
	Now denote by $\left(\gamma \circ (\Omega \otimes\Omega)\right)_{X,Y}^Z$ the $Z$-component of the restriction of $\gamma \circ (\Omega \otimes\Omega)$ to $X^\vee \otimes X \otimes Y^\vee \otimes Y$ (again, we assume $X,Y,Z\in\Proj \cat{C}$). 
	The computation 
	\begin{align}
	\left(\gamma \circ (\Omega \otimes\Omega)\right)_{X,Y}^Z&=	\begin{tikzpicture}
	\begin{pgfonlayer}{nodelayer}
		\node [style=flat box 2] (0) at (-11.5, 7) {{\small $\widetilde \theta_{X,Z}$}};
		\node [style=flat box 2] (1) at (-9.5, 7) {{\small $\widetilde \theta_{Y,Z}$}};
		\node [style=none] (2) at (-12, 5) {};
		\node [style=none] (3) at (-11, 5) {};
		\node [style=none] (4) at (-10, 5) {};
		\node [style=none] (5) at (-9, 5) {};
		\node [style=none] (6) at (-11, 8) {};
		\node [style=none] (7) at (-10, 8) {};
		\node [style=none] (8) at (-12, 9) {};
		\node [style=none] (9) at (-9, 9) {};
		\node [style=none] (10) at (-12, 4.5) {{\small $X^\vee$}};
		\node [style=none] (11) at (-11, 4.5) {{\small $X$}};
		\node [style=none] (12) at (-10, 4.5) {{\small $Y^\vee$}};
		\node [style=none] (13) at (-9, 4.5) {{\small $Y$}};
		\node [style=none] (14) at (-12, 9.5) {{\small $Z$}};
		\node [style=none] (15) at (-9, 9.5) {{\small $Z^\vee$}};
	\end{pgfonlayer}
	\begin{pgfonlayer}{edgelayer}
		\draw (2.center) to (8.center);
		\draw (3.center) to (6.center);
		\draw (4.center) to (7.center);
		\draw (5.center) to (9.center);
		\draw [bend left=90, looseness=2.00] (6.center) to (7.center);
	\end{pgfonlayer}
\end{tikzpicture}\ = \ 	\begin{tikzpicture}
	\begin{pgfonlayer}{nodelayer}
		\node [style=flat box 2] (0) at (-11.5, 7) {{\small $\widetilde \theta_{X,Z}$}};
		\node [style=flat box 2] (1) at (-6, 8) {{\footnotesize $\alpha(Y,Z)$}};
		\node [style=none] (2) at (-12, 5) {};
		\node [style=none] (3) at (-11, 5) {};
		\node [style=none] (4) at (-8, 5) {};
		\node [style=none] (5) at (-6, 5) {};
		\node [style=none] (6) at (-11, 8) {};
		\node [style=none] (7) at (-8, 8) {};
		\node [style=none] (8) at (-12, 9) {};
		\node [style=none] (9) at (-6, 9) {};
		\node [style=none] (10) at (-12, 4.5) {{\small $X^\vee$}};
		\node [style=none] (11) at (-11, 4.5) {{\small $X$}};
		\node [style=none] (12) at (-8, 4.5) {{\small $Y^\vee$}};
		\node [style=none] (13) at (-6, 4.5) {{\small $Y$}};
		\node [style=none] (14) at (-12, 9.5) {{\small $Z$}};
		\node [style=none] (15) at (-6, 9.5) {{\small $Z^\vee$}};
		\node [style=flat box 2] (16) at (-8, 6) {{\footnotesize ${\alpha(Y,Z)''}^\vee$}};
	\end{pgfonlayer}
	\begin{pgfonlayer}{edgelayer}
		\draw (2.center) to (8.center);
		\draw (3.center) to (6.center);
		\draw [in=270, out=90] (4.center) to (7.center);
		\draw (5.center) to (9.center);
		\draw [bend left=90] (6.center) to (7.center);
	\end{pgfonlayer}
\end{tikzpicture}\ = \ 	\begin{tikzpicture}
	\begin{pgfonlayer}{nodelayer}
		\node [style=flat box 2] (0) at (-12, 6) {{\small $\quad\widetilde \theta_{X,Z}\quad$}};
		\node [style=flat box 2] (1) at (-7.5, 7) {{\footnotesize $\alpha(Y,Z)$}};
		\node [style=none] (2) at (-12.75, 5) {};
		\node [style=none] (3) at (-11, 5) {};
		\node [style=none] (4) at (-9, 5) {};
		\node [style=none] (5) at (-7.5, 5) {};
		\node [style=none] (6) at (-11, 8) {};
		\node [style=none] (7) at (-9, 8) {};
		\node [style=none] (8) at (-12.75, 9) {};
		\node [style=none] (9) at (-7.5, 9) {};
		\node [style=none] (10) at (-12.75, 4.5) {{\small $X^\vee$}};
		\node [style=none] (11) at (-11, 4.5) {{\small $X$}};
		\node [style=none] (12) at (-9, 4.5) {{\small $Y^\vee$}};
		\node [style=none] (13) at (-7.5, 4.5) {{\small $Y$}};
		\node [style=none] (14) at (-12.75, 9.5) {{\small $Z$}};
		\node [style=none] (15) at (-7.5, 9.5) {{\small $Z^\vee$}};
		\node [style=flat box 2] (16) at (-11, 7.5) {{\footnotesize ${\alpha(Y,Z)''}$}};
	\end{pgfonlayer}
	\begin{pgfonlayer}{edgelayer}
		\draw (2.center) to (8.center);
		\draw (3.center) to (6.center);
		\draw [in=270, out=90] (4.center) to (7.center);
		\draw (5.center) to (9.center);
		\draw [bend left=90, looseness=1.25] (6.center) to (7.center);
	\end{pgfonlayer}
\end{tikzpicture}\\  \ &= \ 	\begin{tikzpicture}
	\begin{pgfonlayer}{nodelayer}
		\node [style=flat box 2] (0) at (-12, 6) {{\small $\quad \theta_{X,Y}\quad$}};
		\node [style=flat box 2] (1) at (-7.5, 7) {{\footnotesize $\alpha(Y,Z)$}};
		\node [style=none] (2) at (-12.75, 5) {};
		\node [style=none] (3) at (-11, 5) {};
		\node [style=none] (4) at (-9, 5) {};
		\node [style=none] (5) at (-7.5, 5) {};
		\node [style=none] (6) at (-11, 8) {};
		\node [style=none] (7) at (-9, 8) {};
		\node [style=none] (8) at (-12.75, 9) {};
		\node [style=none] (9) at (-7.5, 9) {};
		\node [style=none] (10) at (-12.75, 4.5) {{\small $X^\vee$}};
		\node [style=none] (11) at (-11, 4.5) {{\small $X$}};
		\node [style=none] (12) at (-9, 4.5) {{\small $Y^\vee$}};
		\node [style=none] (13) at (-7.5, 4.5) {{\small $Y$}};
		\node [style=none] (14) at (-12.75, 9.5) {{\small $Z$}};
		\node [style=none] (15) at (-7.5, 9.5) {{\small $Z^\vee$}};
		\node [style=flat box 2] (16) at (-12.75, 7.5) {{\footnotesize ${\alpha(Y,Z)''}^\vee$}};
	\end{pgfonlayer}
	\begin{pgfonlayer}{edgelayer}
		\draw (2.center) to (8.center);
		\draw (3.center) to (6.center);
		\draw [in=270, out=90] (4.center) to (7.center);
		\draw (5.center) to (9.center);
		\draw [bend left=90, looseness=1.25] (6.center) to (7.center);
	\end{pgfonlayer}
\end{tikzpicture} \ = \ 	\begin{tikzpicture}
	\begin{pgfonlayer}{nodelayer}
		\node [style=flat box 2] (0) at (-12.25, 6) {{\small $ \theta_{X,Y}$}};
		\node [style=none] (2) at (-12.75, 5) {};
		\node [style=none] (3) at (-12, 5) {};
		\node [style=none] (4) at (-11, 5) {};
		\node [style=none] (5) at (-10, 5) {};
		\node [style=none] (6) at (-12, 8) {};
		\node [style=none] (7) at (-11, 8) {};
		\node [style=none] (8) at (-12.75, 10.25) {};
		\node [style=none] (9) at (-10, 10.25) {};
		\node [style=none] (10) at (-12.75, 4.5) {{\small $X^\vee$}};
		\node [style=none] (11) at (-12, 4.5) {{\small $X$}};
		\node [style=none] (12) at (-11, 4.5) {{\small $Y^\vee$}};
		\node [style=none] (13) at (-10, 4.5) {{\small $Y$}};
		\node [style=none] (14) at (-12.75, 10.75) {{\small $Z$}};
		\node [style=none] (15) at (-10, 10.75) {{\small $Z^\vee$}};
		\node [style=flat box 2] (17) at (-11.5, 9.25) {{\small $\qquad\widetilde \theta_{Y,Z}\qquad$}};
	\end{pgfonlayer}
	\begin{pgfonlayer}{edgelayer}
		\draw (2.center) to (8.center);
		\draw (3.center) to (6.center);
		\draw [in=270, out=90] (4.center) to (7.center);
		\draw (5.center) to (9.center);
		\draw [bend left=90, looseness=1.25] (6.center) to (7.center);
	\end{pgfonlayer}
\end{tikzpicture} \ = \   \left( \Omega \circ  \ostar   \right)_{X,Y}^Z \ 
	\end{align}
	now proves $\gamma \circ (\Omega \otimes \Omega)=\Omega \circ \ostar$ and hence finishes the proof of the assertion.  
	Note that in the last line the tilde on the $\theta$ disappears because the pivotal structure is absorbed into the dual of the map $\alpha(Y,Z)''$ which, by abuse of notation, we see as a map ${\alpha(Y,Z)''}^\vee : Y^\vee \to Z$.  
\end{proof}

The product $\ostar$ induces the block diagonal product $\star$ on the Hochschild chains in the following sense:

\begin{theorem}\label{thmstarPsiF}
	Let $\cat{C}$ be a finite tensor category with symmetric Frobenius structure. Then the equi\-valence $\lint^{X\in\Proj\cat{C}}\cat{C}(X,X)\simeq  \cat{C}(I,\mathbb{F}_\bullet)$ of differential graded vector spaces from Proposition~\ref{corresocoend2}
	yields an equi\-valence 
	$
	\bspace{ \lint^{X\in\Proj\cat{C}}\cat{C}(X,X) }{\star}\simeq \bspace{ \cat{C}(I,\mathbb{F}_\bullet)  }{\ostar_\bullet} 
	$ of non-unital $E_2$-algebras.
\end{theorem}

\begin{proof}
	We already know that the product on the left hand side, up to homotopy, is supported in degree zero. 
	In fact, this can also be directly seen for $\bspace{ \cat{C}(I,\mathbb{F}_\bullet)  }{\ostar_\bullet}$.
	Hence, it remains to confirm the compatibility of $\lint^{X\in\Proj\cat{C}}\cat{C}(X,X)\simeq  \cat{C}(I,\mathbb{F}_\bullet)$ with the algebra structure in degree zero.
	For this purpose, let us denote the equivalence $\lint^{X\in\Proj\cat{C}}\cat{C}(X,X)\simeq  \cat{C}(I,\mathbb{F}_\bullet)$
	by $Z$.
	If we choose $\flint^{X\in \Proj\cat{C}} X^\vee \otimes X$ as the resolution of $\mathbb{F}$ (Proposition~\ref{corresocoend2}) and endomorphisms $f:P\to P$ and $g:Q\to Q$ for $P,Q\in\Proj\cat{C}$,
	we can obtain with Sweedler notation $\Delta_{P,Q} (\id_P)=\alpha' \otimes \alpha'' \in \cat{C}(P,Q)\otimes\cat{C}(Q,P)$ 
	\begin{align}
	Z(f) \ostar Z(g) \ =\ \begin{tikzpicture}
	\begin{pgfonlayer}{nodelayer}
		\node [style=none] (0) at (-12, 3) {};
		\node [style=none] (1) at (-10, 3) {};
		\node [style=none] (2) at (-9, 3) {};
		\node [style=none] (3) at (-8, 3) {};
		\node [style=flat box 2] (4) at (-10, 3.5) {{\small $f$}};
		\node [style=flat box 2] (5) at (-8, 3.75) {{\small $g$}};
		\node [style=flat box 2] (6) at (-10, 4.75) {{\small $\alpha'$}};
		\node [style=flat box 2] (7) at (-12, 4.75) {{\small ${\alpha''}^\vee$}};
		\node [style=none] (8) at (-10, 5.5) {};
		\node [style=none] (9) at (-9, 5.5) {};
		\node [style=none] (10) at (-8, 6.5) {};
		\node [style=none] (11) at (-12, 6.5) {};
		\node [style=none] (12) at (-12, 7) {{\small $Q^\vee$}};
		\node [style=none] (13) at (-8, 7) {{\small $Q$}};
	\end{pgfonlayer}
	\begin{pgfonlayer}{edgelayer}
		\draw [bend right=90] (0.center) to (1.center);
		\draw [bend right=90, looseness=1.75] (2.center) to (3.center);
		\draw (3.center) to (10.center);
		\draw (2.center) to (9.center);
		\draw [bend right=90, looseness=1.50] (9.center) to (8.center);
		\draw (8.center) to (1.center);
		\draw (0.center) to (11.center);
	\end{pgfonlayer}
\end{tikzpicture} \ =\ \begin{tikzpicture}
	\begin{pgfonlayer}{nodelayer}
		\node [style=none] (2) at (-10, 3.75) {};
		\node [style=none] (3) at (-8, 3.75) {};
		\node [style=flat box 2] (4) at (-8, 5.75) {{\small $f$}};
		\node [style=flat box 2] (5) at (-8, 8.25) {{\small $g$}};
		\node [style=flat box 2] (6) at (-8, 7) {{\small $\alpha'$}};
		\node [style=flat box 2] (7) at (-8, 4.5) {{\small ${\alpha''}$}};
		\node [style=none] (9) at (-10, 9) {};
		\node [style=none] (10) at (-8, 9) {};
		\node [style=none] (13) at (-8, 9.5) {{\small $Q$}};
		\node [style=none] (14) at (-10, 9.5) {{\small $Q^\vee$}};
	\end{pgfonlayer}
	\begin{pgfonlayer}{edgelayer}
		\draw [bend right=90, looseness=0.75] (2.center) to (3.center);
		\draw (3.center) to (10.center);
		\draw (2.center) to (9.center);
	\end{pgfonlayer}
\end{tikzpicture} \ = \ Z(g\star f) \simeq Z(f\star g) \ . 
	\end{align}
	The first equality follows from
	the proof of Proposition~\ref{propmultostar}, 
	the third one is \cite[Lemma~5.2]{trace} in Sweedler notation.
\end{proof}

\section{The differential graded Verlinde formula\label{secverlindeformula}}
In this section, we combine the preparations of the previous sections with the following main result of \cite{E2} which finally affords the relation to Deligne's Conjecture:
For any finite tensor category $\cat{C}$, the canonical algebra $\mathbb{A} \in \cat{C}$ lifts to an algebra $\alg \in Z(\cat{C})$ in the Drinfeld center
(Section~\ref{secrevewidmf}), and
Davydov, Müger, Nikshych,  Ostrik prove in
\cite{dmno} that $\alg$ is in fact braided commutative. By Theorem~\ref{thmE2derivedhom} this implies that $\cat{C}(I,\mathbb{A}^\bullet)$ becomes an $E_2$-algebra. By means of Proposition~\ref{propcocomplex}, this means that the Hochschild cochain complex of $\cat{C}$ inherits an $E_2$-structure. The extremely crucial insight is that this $E_2$-structure is a solution to Deligne's Conjecture for $\Proj \cat{C}$, as linear category. Since the $E_2$-structure on $\cat{C}(I,\mathbb{A}^\bullet)$ is at least formulated in terms of the monoidal structure of $\cat{C}$, this is a very non-obvious statement:

\begin{theorem}[Comparison Theorem $\text{\cite[Theorem~5.1]{E2}}$]\label{thhmcomparisondeligne}
	For any finite tensor category $\cat{C}$, 
	the algebra structure on the canonical end $\mathbb{A}=\int_X X\otimes X^\vee$ induces  an $E_2$-algebra structure on the 
	homotopy invariants $\cat{C}(I,\mathbb{A}^\bullet)$. Under the equivalence $\cat{C}(I,\mathbb{A}^\bullet)\simeq \rint_{X\in \Proj \cat{C}} \cat{C}(X,X)$ from Proposition~\ref{propcocomplex}, this $E_2$-structure
	provides a solution to Deligne's Conjecture in the sense that it induces the standard Gerstenhaber structure on the Hochschild cohomology of $\cat{C}$.
\end{theorem}

Having described Deligne's $E_2$-structure
as the homotopy invariants of a braided commutative algebra,
we can now prove our first main result:

\begin{theorem}[Differential graded Verlinde formula for the Hochschild cochain complex]\label{thmdgva1}
	For any modular category $\cat{C}$, the action of the mapping class group element  $S^{-1}=\begin{pmatrix} \phantom{-}0 & 1 \\ -1 &0\end{pmatrix}\in \SL(2,\mathbb{Z})$ on the Hochschild cochain complex of $\cat{C}$ 
	yields an equivalence
	\begin{align} \mfd(S^{-1}):    \bspace{  \rint_{X\in \Proj \cat{C}} \cat{C}(X,X)  }{\otimes} \simeq \bspace{ \rint_{X\in \Proj \cat{C}} \cat{C}(X,X)  }{\smile} \label{eqnSinvequiv}
	\end{align} 
	of $E_2$-algebras which are given as follows:\begin{itemize}
		\item On the left hand side, the $E_2$-structure is the differential graded Verlinde algebra on the Hochschild cochains of $\cat{C}$
		induced by the monoidal product
		(Theorem~\ref{thmverlindecochain}).
		\item On the right hand side, the $E_2$-structure is the one afforded by Deligne's Conjecture with the underlying multiplication being the cup product $\smile$.
	\end{itemize}
\end{theorem}

\begin{proof}
	According to the definition
	of the \emph{dual} differential graded modular functor, the mapping class group element $S^{-1}$ acts on the dual conformal block by acting with $S$ on the chain version of the differential graded conformal block and dualization.  From Proposition~\ref{propSmatrix}, we can  
	therefore conclude  that  the action of $S^{-1}$ on $\rint_{X\in \Proj \cat{C}} \cat{C}(X,X)$ is given by the composition of equivalences
	\begin{align} \rint_{X\in \Proj \cat{C}} \cat{C}(X,X)  \simeq \cat{C}(I,\mathbb{A}^\bullet) \ra{\Psi^\bullet} \cat{C}(I,\mathbb{F}^\bullet) \ra{\drinfeld^\bullet} \cat{C}(I,\mathbb{A}^\bullet) \simeq \rint_{X\in \Proj \cat{C}} \cat{C}(X,X) \ , \label{eqnalltheequiv}
	\end{align} where both unlabeled equivalences are the canonical one from Proposition~\ref{propcocomplex} (recall that it can be turned into an isomorphism for suitable models of the homotopy end).
	The isomorphism $\Psi : \mathbb{A}\to\mathbb{F}$ is the Radford map
	and $\drinfeld:\mathbb{F}\to\mathbb{A}$ is the Drinfeld map.	
	We can now consider the following diagram in which the vertices are $E_2$-algebras:
	\begin{equation}
	\begin{tikzcd}
	\ar[rrrr,"\simeq"] \ar[dd,swap,"\mfd(S^{-1})"] \bspace{\rint_{X\in\Proj\cat{C}} \cat{C}(X,X)  }{\otimes}   &&& & \bspace{\cat{C}(I,\mathbb{A}^\bullet)  }{\otimes}   \ar[d,"\Psi^\bullet"]  \\    &&&& \ar[d,"\drinfeld^\bullet"]\mathfrak{A}_\cat{C}\dolph \\  \bspace{\rint_{X\in\Proj\cat{C}} \cat{C}(X,X)  }{\cup}     \ar[rrrr,swap,"\simeq"] &&&&    \bspace{\cat{C}(I,\mathbb{A}^\bullet)  }{\gamma^\bullet}      \\ 
	\end{tikzcd}
	\end{equation}
	The description of $S^{-1}$ that we have just extracted from Proposition~\ref{propSmatrix} means that the diagram commutes as a diagram of chain complexes.
	It remains to be shown that all of the maps in the diagram, except $\mfd(S^{-1})$, are not only chain maps, but maps of $E_2$-algebras because then $\mfd(S^{-1})$ is also a map of $E_2$-algebras.
	For all of the maps appearing in the diagram,
	this has been established previously in the text; we just have to tie everything together:
	The upper horizontal map and $\Psi^\bullet$ are maps of $E_2$-algebras as follows from the construction of $\bspace{\rint_{X\in\Proj\cat{C}} \cat{C}(X,X)  }{\otimes}$
	in Theorem~\ref{thmverlindecochain}. For $\drinfeld^\bullet$, this is a consequence of Proposition~\ref{propdolphdrinfeld}. Finally, for the lower horizontal map, this is exactly the content of Theorem~\ref{thhmcomparisondeligne}.
	This finishes the proof.
\end{proof}

\begin{corollary}[Framed extension]\label{corframed}
	If $\cat{C}$ is a unimodular finite ribbon category 
	and if a trivialization $D\cong I$ has been fixed, the differential graded Verlinde algebra $\bspace{  \rint_{X\in \Proj \cat{C}} \cat{C}(X,X)  }{\otimes}$ naturally extends to a framed $E_2$-algebra.
	If additionally we fix for $\bspace{ \rint_{X\in \Proj \cat{C}} \cat{C}(X,X)  }{\smile}$ the framed $E_2$-extension afforded by the balancing of $Z(\cat{C})$, and if $\cat{C}$ is modular, $S^{-1}$ acts by an equivalence of framed $E_2$-algebras.
\end{corollary}

\begin{proof}
	One can show that the
	balancing $\theta$ at $\mathbb{F}\dolph$ is the identity
	(a proof of this fact is given in \cite[Lemma~2.10~(i)]{cardycase} under slightly stronger assumptions, but the argument applies here as well; it uses that $\cat{C}$ is actually ribbon and not just balanced).
	Now we can conclude from
	Theorem~\ref{thmE2derivedhom}
	that the dolphin algebra of $\cat{C}$ extends to a framed $E_2$-algebra.
	But then $\bspace{  \rint_{X\in \Proj \cat{C}} \cat{C}(X,X)  }{\otimes}$, by its very construction in Theorem~\ref{thmverlindecochain},  becomes a framed $E_2$-algebra as well.
	The fact that $\bspace{ \rint_{X\in \Proj \cat{C}} \cat{C}(X,X)  }{\smile}$ extends to a framed $E_2$-algebra through the balancing of $Z(\cat{C})$ is the chain level generalization of Menichi's Theorem \cite[Theorem~63]{menichi} given in \cite[Corollary~7.4]{E2}.
	In order to see that $S^{-1}$, in the modular case, respects also the \emph{framed} $E_2$-structures follows now by the same line of argument as in the proof of Theorem~\ref{thmdgva1}.
\end{proof}

For a cohomology class $[\varphi]$ in the Hochschild cochain complex of a modular category $\cat{C}$, we write the action by 
the mapping class group element as $S[\varphi]$ (instead of $\mfd(S)[\varphi]$).
Then Theorem~\ref{thmdgva1} tells us in particular
\begin{align}
S  [\varphi]    \otimes S[\psi]	 = S ([\varphi] \cup [\psi]) \ , \quad
\left[ S [\varphi],S[\psi] \right] _\otimes = S\left[ [\varphi],[\psi]  \right] \ , \label{eqnGerstenhaberbracket}
\end{align}
where $\otimes$, by slight abuse of notation, denotes the multiplication induced by the monoidal product in the sense of Theorem~\ref{thmverlindecochain}; moreover, we denote by $[-,-]_\otimes$ the Gerstenhaber bracket associated to $\otimes$ and by $[-,-]$ the usual Gerstenhaber bracket on Hochschild cohomology.

\begin{example}\label{exbracketnonzero}
	Consider the modular category $\Mod_k D(G)$ of 
	finite-dimensional modules over the Drinfeld double of 
	a finite group $G$ (see also Example~\ref{exhopfalgebra}).
	Then the differential graded modular functor for $\Mod_k D(G)$
	can be seen as a differential graded version of the Dijkgraaf-Witten modular functor as explained in~\cite[Example~3.13]{dmf}.
	Dualizing Example~\ref{exhopfalgebra}, the dual differential graded conformal block for the torus, i.e.\ the Hochschild cochain complex of $D(G)$,
	is equivalent  to the complex $C^*(\PBun_G(\mathbb{T}^2);k)$ of cochains on the groupoid of $G$-bundles over the torus.
	Now the cohomology of the differential graded Verlinde algebra of $\Mod_k D(G)$, seen as Batalin-Vilkovisky algebra,
	is determined by the Batalin-Vilkovisky structure on the Hochschild cohomology of group algebras and the mapping class group action on $C^*(\PBun_G(\mathbb{T}^2);k)$ (which is the geometric one).

	An example for the non-triviality of the Gerstenhaber bracket of the differential graded Verlinde algebra can be obtained as follows:
	Over an algebraically closed field of characteristic $p>0$,
	the differential graded Verlinde algebra of modules over $D(\mathbb{Z}_p)$ 
	has a non-zero Gerstenhaber bracket. 	
	In order to see this, observe that the linear category of modules over $D(\mathbb{Z}_p)$ is equivalent to modules over the action groupoid $\mathbb{Z}_p // \mathbb{Z}_p$ of the conjugation action of $\mathbb{Z}_p$ on itself, which is trivial here, of course. Therefore, $\mathbb{Z}_p // \mathbb{Z}_p\simeq \sqcup_{\mathbb{Z}_p}  \star // \mathbb{Z}_p$. 
	Now the statement follows from~\eqref{eqnGerstenhaberbracket}
	and the computation of the Gerstenhaber bracket on $HH^*(k[\mathbb{Z}_p])$ in \cite{lezhou}, where it is shown in particular that the Gerstenhaber bracket is non-trivial. 
	
	Thanks to Theorem~\ref{thmdgva1}, the statement that the cohomology of the differential graded Verlinde algebra can be obtained through the Hochschild cohomology (which can be seen as a Gerstenhaber algebra or Batalin-Vilkovisky algebra) and the $\SL(2,\mathbb{Z})$-action remains true beyond Drinfeld doubles. Here, however, obtaining the needed ingredients is much more involved. 
	At least in the general Hopf-algebraic case, the mapping class group action is explicitly given in \cite{svea} (see also the comments in Example~\ref{exhopfalgebra}).
	The Hochschild cohomology, at least as graded ring, is known e.g.\ for certain small quantum groups \cite{lq}.
	A further investigation of this class of examples is beyond the scope of this article.
\end{example}

Spelling out the Verlinde formula on Hochschild cochains (Theorem~\ref{thmdgva1})
in zeroth cohomology, 
we recover a formula that Gainutdinov and Runkel have proposed and proven in \cite{grv} as a non-semisimple generalization of the Verlinde formula.
Their result is partly phrased in terms of the linear Grothendieck ring: Recall from \cite[Definition~4.5.2]{egno}
that for a finite tensor category $\cat{C}$, the \emph{Grothendieck ring} $\catf{Gr}\, \cat{C}$ 
of $\cat{C}$ is the free abelian group generated by a complete set of representatives $(X_i)_{i=0,\dots,n}$ for its finitely many isomorphism classes of simple objects (we denote the generator corresponding to $X_i$ by $[X_i]$), where the ring structure is given by
$
[X_i] \cdot [X_j] := \sum_{\ell=0}^n N_{ij}^\ell [X_\ell]
$
with $N_{ij}^\ell:= [X_i\otimes X_j:X_\ell]\in\mathbb{N}_0$
being the multiplicity of the simple object $X_\ell$ in the Jordan-Hölder series of the tensor product $X_i\otimes X_j$ (the numbers $N_{ij}^\ell$ generalize the fusion coefficients used in the semisimple case). 
Let now $\cat{C}$ be pivotal and unimodular. Then by \cite[Theorem~4.1 \& Corollary~4.3]{shimizucf} the \emph{internal character map}
\begin{align}
\catf{ch} : \catf{Gr}_k \cat{C} = k\otimes_\mathbb{Z} \catf{Gr}\, \cat{C} \to \catf{CF}(\cat{C})\ , \quad [X_i] \mapsto \left(  I \ra{\widetilde b_X} {^\vee X} \otimes X \ra{\substack{\text{pivotal}\\ \text{structure}}} X^\vee \otimes X \to\mathbb{F}   \right)
\end{align}
exhibits the linear Grothendieck ring of $\cat{C}$ as a subalgebra of the algebra $\catf{CF}(\cat{C}) = \cat{C}(I,\mathbb{F})$
of class functions. 

If $\Psi : \mathbb{A}\to\mathbb{F}$ is again the Radford map, the family 
$
(\phi_i)_{i=0,\dots,n}$ ,
where $\phi_i :=     \Psi^{-1} \circ \catf{ch}(X_i)   : I \to \mathbb{A} $,
is linear independent in $\cat{C}(I,\mathbb{A})$. 
For the next statement, we will denote the automorphism of $\cat{C}(I,\mathbb{A})$
corresponding to the action of $S^{-1}$ on $HH^0(\cat{C})\cong\cat{C}(I,\mathbb{A})$ by $\mathfrak{S}$ (we do this to match the slightly different conventions in \cite{grv}).
Moreover, we will denote the multiplication on $\cat{C}(I,\mathbb{A})$ coming from the cup product by $\circ$ because it amounts to the composition of natural endotransformations of the identity functor of $\cat{C}$.

\begin{corollary}[$\text{Gainutdinov-Runkel \cite[Theorem~3.9]{grv}}$]\label{corgrv}
	Let $\cat{C}$ be a modular category and $(\phi_i)_{i=0,\dots,n	}$ the linear independent family associated  
	to a complete set of representatives of the finitely many isomorphism classes of simple objects.
	Then
	\begin{align}
	\mathfrak{S}^{-1} \left(	\mathfrak{S}(\phi_i) \circ \mathfrak{S}(\phi_j)\right) = \sum_{\ell=0}^n N_{ij}^\ell \phi_\ell \ . 
	\end{align}
\end{corollary}

In the semisimple case, this statement reduces to the ordinary Verlinde formula.

\begin{proof}
	Theorem~\ref{thmdgva1}, when spelled out in zeroth cohomology,
	states  that the zeroth cohomology restriction of the action of $S^{-1}$, namely the map
	$
	\mathfrak{S} 
	:	\cat{C}(I,\mathbb{A}) \ra{ \Psi_*  } \cat{C}(I,\mathbb{F}) \ra{\drinfeld_*} \cat{C}(I,\mathbb{A})$
	induced by the Radford map $\Psi : \mathbb{A}\to\mathbb{F}$
	and the Drinfeld map $\drinfeld:\mathbb{F}\to\mathbb{A}$ is an isomorphism of algebras if we endow \begin{itemize}
		\item the vector space $\cat{C}(I,\mathbb{A})$  on the left hand side with the multiplication from Theorem~\ref{thmverlindecochain},
		\item and the vector space $\cat{C}(I,\mathbb{A})$ on the right hand side with the multiplication coming from the cup product on zeroth Hochschild cohomology 
		(which here is just the multiplication coming from the usual algebra structure $\gamma : \mathbb{A}\otimes\mathbb{A} \to \mathbb{A}$). 
	\end{itemize}
	Now the map
	$	\cat{C}(I,\mathbb{F}) \ra{\Psi^{-1}_*}    \cat{C}(I,\mathbb{A}) \ra{\mathfrak{S}}  \cat{C}(I,\mathbb{A}) 	$
	is an isomorphism of algebras if $\cat{C}(I,\mathbb{F})$ is endowed with the product coming from the multiplication $\mu:\mathbb{F}\otimes\mathbb{F}\to\mathbb{F}$ defined using the braiding of $\cat{C}$, see Proposition~\ref{propoFmult}.
	Recall that by  Remark~\ref{remrelclassfunction} the algebra $\cat{C}(I,\mathbb{F})$ actually agrees with the algebra of class functions of $\cat{C}$. In summary,  Theorem~\ref{thmdgva1}, when evaluated in zeroth cohomology, states that
	\begin{align}
	\mathfrak{S}\circ \Psi_*^{-1}:	\catf{CF}(\cat{C}) \ra{\Psi^{-1}_*}    \cat{C}(I,\mathbb{A}) \ra{\mathfrak{S}}  (\cat{C}(I,\mathbb{A}),\gamma_*)=(\cat{C}(I,\mathbb{A}),\circ) \label{eqncomppsis}
	\end{align}
	is an isomorphism of algebras.
	(Of course, when considering the composition~\eqref{eqncomppsis}, we can actually cancel $\Psi$, so that the statement that \eqref{eqncomppsis} is an isomorphism of algebras will alternatively follow from Proposition~\ref{propdrinfeldmapmorphismus}, but we actually need the factorization~\eqref{eqncomppsis} to compare to \cite{grv}.)
	With the definition of the family $(\phi_i)_{i=0,\dots,n}$, we find:
	\begin{align}
	\mathfrak{S}(\phi_i) \circ \mathfrak{S}(\phi_j) &= \mathfrak{S} (\phi_i \otimes \phi_j ) = \mathfrak{S} \circ \Psi_*^{-1} (\catf{ch} X_i \cdot \catf{ch} X_j) \\ &= \mathfrak{S} \circ \Psi_*^{-1} \left(    \sum_{\ell=0}^n N_{ij}^\ell\,  \catf{ch} X_\ell\right) 
	= \mathfrak{S}   \left(   \sum_{\ell=0}^n N_{ij}^\ell \phi_\ell\right) \ . 
	\end{align}
\end{proof}

We now prove our second main result. It is concerned with the effect of the $S$-transformation on the products on the Hochschild \emph{chain} complex.
As in the case of Theorem~\ref{thmdgva1}, the mapping class group action comes  from the differential graded modular functor that $\cat{C}$ gives rise to.

\begin{theorem}[Differential graded Verlinde formula for the Hochschild chain complex]\label{thmdgva2}
	For any modular category $\cat{C}$, the 
	action of the mapping class group element  $S=\begin{pmatrix} 0 & -1 \\ 1 &\phantom{-}0\end{pmatrix}\in \SL(2,\mathbb{Z})$ yields an equivalence
	\begin{align}  \mfc( S): \bspace{  \lint^{X\in \Proj \cat{C}} \cat{C}(X,X)  }{\otimes} \simeq \bspace{ \lint^{X\in \Proj \cat{C}} \cat{C}(X,X)  }{\star } 
	\end{align} 
	of non-unital $E_2$-algebras whose multiplication, up to homotopy, is concentrated in degree zero. \begin{itemize}
		\item On the left hand side, the $E_2$-structure is the differential graded Verlinde algebra on the Hochschild chains of $\cat{C}$
		induced the monoidal product \cite{dva}, see Proposition~\ref{propdgva}.
		
		\item On the right hand side, the non-unital $E_2$-structure is the almost trivial one that is a part of the cyclic version of Deligne's Conjecture applied to the Calabi-Yau structure coming from the modified trace
		on the tensor ideal of projective objects.
	\end{itemize}
\end{theorem}

\begin{proof}
	Similarly to the  proof of Theorem~\ref{thmdgva1}, all ingredients have been  established, and we just tie them together:
	The effect of the mapping class group element $S$ was computed in Proposition~\ref{propSmatrix}:
	After the canonical identification $\lint^{X\in\Proj\cat{C}}\cat{C}(X,X)\simeq \cat{C}(I,\mathbb{F}_\bullet)$, it acts as the equivalence
	$
	\cat{C}(I,\pF) \ra{\drinfeld_\bullet }	\cat{C}(I,\pA) \ra{\Psi_\bullet} \cat{C}(I,\pF) $. 
	In fact, these maps are morphisms of non-unital $E_2$-algebras
	\begin{align}
	\bspace{	\cat{C}(I,\pF)}{\otimes} \ra{\drinfeld_\bullet}  \bspace{	\cat{C}(I,\pA)}{\gamma_\bullet} \ra{\Psi_\bullet} \bspace {\cat{C}(I,\pF)}{\ostar}  \ .  
	\end{align}
	This is	a consequence of 
	Proposition~\ref{propdrinfeldmapmorphismus} for $\drinfeld_\bullet$.
	For $\Psi_\bullet$, it follows from	Proposition~\ref{propmultostar}. 
	It remains to confirm that under the equivalence     $\lint^{X\in\Proj\cat{C}}\cat{C}(X,X)\simeq \cat{C}(I,\mathbb{F}_\bullet)$, the non-unital $E_2$-algebra $\bspace {\cat{C}(I,\pF)}{\ostar} $ translates into the non-unital $E_2$-algebra afforded by the cyclic Deligne Conjecture applied to the modified trace on the tensor ideal of projective objects.
	Indeed, this follows from $\bspace {\cat{C}(I,\pF)}{\ostar}\simeq \bspace{ \lint^{X\in \Proj \cat{C}} \cat{C}(X,X)  }{\star } $  (Theorem~\ref{thmstarPsiF}) and the fact that $\star$ is actually the non-unital $E_2$-multiplication coming from the cyclic Deligne Conjecture applied to the modified trace. The latter is a consequence of the results of \cite{trace}
	and in particular Theorem~\ref{thmmtrace} from above.
\end{proof}

\begin{remark}[Products versus coproducts]\label{remproductcoproduct}
	There seems to be an asymmetry between 
	Theorem~\ref{thmdgva1} on Hochschild cochains, 
	where two rather rich higher multiplicative structures are compared, and
	Theorem~\ref{thmdgva2} on Hochschild chains, which is concerned with an almost trivial product.  
	This asymmetry, however, is mostly a consequence of our way of presenting the results. If we formulated statements about \emph{coproducts} instead of products, the situation would be reversed.
\end{remark}

\begin{corollary}\label{correducessi}
	For a semisimple modular category $\cat{C}$,
	the statements of Theorem~\ref{thmdgva1} and~\ref{thmdgva2} 
	are equivalent and both amount precisely to the semisimple Verlinde formula.
\end{corollary}

\begin{proof}
	We choose a complete system $x_0=I,x_1,\dots,x_n$ of simple objects of $\cat{C}$ and denote by $[x_i]\in HH_0(\cat{C})$ the  element corresponding to the identity on $x_i$ in zeroth Hochschild homology. Then $HH_0(\cat{C})$ has $[x_0],\dots,[x_n]$ as its basis thanks to $HH_0(\cat{C})\cong \bigoplus_{i=0}^n \cat{C}(X_i,X_i)\cong \bigoplus_{i=0}^n k \cdot \id_{X_i}$.
	We denote by $\Psi_*$ the isomorphism $HH_0(\cat{C})\to HH_0(\cat{C})$ induced by the Radford map $\Psi : \mathbb{A}\to\mathbb{F}$  (here we fix $\mathbb{A}=\mathbb{F}=\bigoplus_{i=0}^n X_i^\vee \otimes X_i$ as a model for both the canonical coend and canonical end). From the concrete description of the inverse $\Omega$ of $\Psi$ (in particular \eqref{eqnpsiinvy} in the proof of Lemma~\ref{lemmapsiomega}), we extract $\Psi_*[x_i]=d_i [x_i]$, where $d_i$ is the usual quantum dimension of $x_i$ (because the modified trace reduces to the quantum trace in the semisimple case).  
	Now we can compute the effect of the $S$-transformation by
	\begin{align}
	\mfc(S)[x_i] \stackrel{\text{Proposition~\ref{propSmatrix}}}{=} \Psi_* \mathbb{D}_* [x_i] =  \sum_{j=0}^n d_j \cdot \begin{tikzpicture}
	\begin{pgfonlayer}{nodelayer}
		\node [style=none] (0) at (-2, 0) {};
		\node [style=none] (2) at (-1.5, 1) {};
		\node [style=none] (3) at (-2, 1) {};
		\node [style=none] (6) at (-2.5, 1) {};
		\node [style=none] (8) at (-0.75, 0.75) {\mlabel{X_i}};
		\node [style=none] (9) at (-2, -0.5) {\mlabel{X_j}};
		\node [style=none] (10) at (-2, 2) {};
		\node [style=none] (11) at (-2, 0) {};
	\end{pgfonlayer}
	\begin{pgfonlayer}{edgelayer}
		\draw [bend right=90, looseness=1.75] (6.center) to (2.center);
		\draw [style=over] (11.center) to (3.center);
		\draw (3.center) to (10.center);
		\draw [style=over, bend left=90, looseness=1.75] (6.center) to (2.center);
	\end{pgfonlayer}
\end{tikzpicture} = \sum_{j=0}^n \begin{tikzpicture}
	\begin{pgfonlayer}{nodelayer}
		\node [style=none] (0) at (-2, 0) {};
		\node [style=none] (1) at (-1.25, 0) {};
		\node [style=none] (2) at (-1.5, 0.5) {};
		\node [style=none] (3) at (-1.75, 0.5) {};
		\node [style=none] (4) at (-2, 1) {};
		\node [style=none] (5) at (-1.25, 1) {};
		\node [style=none] (6) at (-2.5, 0.5) {};
		\node [style=none] (7) at (-0.75, 0.5) {};
		\node [style=none] (8) at (-2.25, -0.5) {\mlabel{X_i}};
		\node [style=none] (9) at (-1, -0.5) {\mlabel{X_j}};
	\end{pgfonlayer}
	\begin{pgfonlayer}{edgelayer}
		\draw [bend right=315, looseness=0.75] (1.center) to (3.center);
		\draw [bend right=45] (2.center) to (4.center);
		\draw [style=over, bend right=45] (0.center) to (2.center);
		\draw [style=over, bend left=45] (3.center) to (5.center);
		\draw [in=180, out=90] (6.center) to (4.center);
		\draw [bend right=45] (6.center) to (0.center);
		\draw [bend right=45] (1.center) to (7.center);
		\draw [bend left=45] (5.center) to (7.center);
	\end{pgfonlayer}
\end{tikzpicture}\cdot [X_j] \ , 
	\end{align} i.e.\ it reduces to the $S$-matrix from the introduction.
	The product on zeroth Hochschild homology induced by the monoidal product in the sense of Proposition~\ref{propdgva} is just given by $[x_i]\otimes [x_j] = \sum_{\ell=0}^n N_{ij}^\ell [x_\ell]$ if $x_i\otimes x_j \cong \bigoplus_{\ell=0}^n N_{ij}^\ell x_\ell$; this can be observed directly or concluded from the description of this product in \cite{dva}. 
	Therefore, Theorem~\ref{thmdgva2} is equivalent to the semisimple Verlinde formula in the formulation~\eqref{eqnverlindeisoalg} if we can show that the multiplication $\star$ in Theorem~\ref{thmdgva2} (which was defined in~\eqref{eqnstarproduct})
	agrees with the $\star$-product $[x_i]\star [x_j]=d_i^{-1}\delta_{i,j} [x_i]$
	from~\eqref{eqnstarprod0}. Indeed, this follows from \cite[Theorem~5.6~(iii)]{trace} (again because the modified dimension agrees with the quantum dimension in the semisimple case).
	
	In order to see that Theorem~\ref{thmdgva1} (the cochain version) also reduces to the semisimple Verlinde formula, we 
	can perform a similar computation. Alternatively, we can observe that in the semisimple case,
	Theorem~\ref{thmdgva1} completely reduces to the statement extracted from it in
	Corollary~\ref{corgrv}, where we reproduced the result from~\cite{grv}. This statement, on the other hand, 
	is equivalent to the usual Verlinde formula in the semisimple case as explained in \cite{grv}.	 
\end{proof}

\label{discussblockdiagonal}
In the semisimple case, the
$S$-transformation transforms the multiplication induced by the monoidal product into a diagonal product $\star$ given in \eqref{eqnstarprod0}, where diagonal means $[X_i]\star [X_j]=0$ if $X_i$ and $X_j$ are non-isomorphic simple objects, i.e.\ if $\cat{C}(X_i,X_j)=0$.
In the non-semisimple case, Theorem~\ref{thmdgva2} achieves at least a \emph{block diagonalization}	 because the product $\star$ is block diagonal \cite[Proposition~5.3]{trace}.
In order to be more explicit, denote by $P_0,\dots,P_n$ a complete system of mutually non-isomorphic indecomposable projective objects of $\cat{C}$ and set $G:=\bigoplus_{i=0}^n P_j$. The object $G$ is a projective generator. We now define $B_1,\dots,B_m$ as the equivalence classes of the equivalence relation $P_i\simeq  P_j :\Leftrightarrow \cat{C}(P_i,P_j) \neq 0$ on $\{P_0,\dots, P_n\}$. We refer to these equivalence classes as \emph{blocks}.
The endomorphism algebra $A:= \cat{C}(G,G)$ 
allows us to write $\cat{C}$, as a linear category, as finite-dimensional modules over $A$. Moreover, $A$
becomes a symmetric Frobenius algebra via the modified trace. In the same way, the endomorphism algebras $A_\ell := \cat{C}(G_\ell,G_\ell)$ of $G_\ell:= \bigoplus_{P_i\in B_\ell} P_i$ for $1\le \ell\le m$
become symmetric Frobenius algebras, and we find
$
A\cong A_1 \oplus \dots \oplus A_m$ as symmetric Frobenius algebras. 
The Hochschild complex $\lint^{X\in\Proj\cat{C}}\cat{C}(X,X)$ is equivalent to the ordinary Hochschild complex of $A$.
In degree zero, i.e.\ on $A$, the product $\star$
from~\eqref{eqnstarproduct}
is given by
\begin{align}
a \star b = a' b a'' \quad \text{for}\quad a,b \in A        \label{starproducteqn}
\end{align} with Sweedler notation $\Delta a = a'\otimes a''$ (note that $\star$ yields only a commutative associative multiplication on $HH_0(\cat{C})=A/[A,A]$, but not on $A$); this follows from \cite{wahlwesterland} or also \cite[Lemma~5.1]{trace}.
The operation being block diagonal now means exactly that it preserves the decomposition
$	A\cong A_1 \oplus \dots \oplus A_m$ in the sense $A_\ell \star A_\ell \subset A_\ell$ and $A_\ell \star A_{\ell'}=0$ for $\ell\neq \ell'$.

A tensor product of the indecomposable projective  objects $P_0,\dots,P_n$ may be decomposed
according to $ P_i \otimes P_j \cong \bigoplus_{\ell=0}^n P_\ell^{\oplus M_{ij}^\ell}$, 
where the multiplicities $M_{ij}^\ell \in \mathbb{N}_0$ are the structure constants of the ring $K_0(\cat{C})$ that as an abelian group is generated by $[P_0],\dots,[P_n]$. 
Via the map \begin{align}K_0(\cat{C})\otimes_\mathbb{Z} k \ra{ [P_i] \mapsto \id_{P_i}  } \bigoplus_{j=0}^n \cat{C}(P_i,P_i)\to HH_0(\cat{C}) \ , 
\end{align}$[P_i]$
gives rise to a class in $HH_0(\cat{C})$ that we denote by $h_i \in HH_0(\cat{C})\cong A / [A,A]$.
If we act with the $S$-transformation on $h_i$, we may represent the result by an element $\mathfrak{s}_i \in A$; the choice we are making here is unique up to commutator. 

\begin{corollary}\label{corKzero}
	With the above notation,
	$	\mathfrak{s}_i' \mathfrak{s}_j \mathfrak{s}_i'' =  \sum_{\ell=0}^n M_{ij}^\ell \mathfrak{s}_\ell \ \text{mod}\  [A,A] $, 
	where $\Delta \mathfrak{s}_i = \mathfrak{s}_i' \otimes \mathfrak{s}_i''$ is the Sweedler notation for the coproduct of the Frobenius structure on $A$ coming from the modified trace.
\end{corollary}

\begin{proof}
	Modulo $[A,A]$, we find
	$	\sum_{\ell=0}^n M_{ij}^\ell \mathfrak{s}_\ell
	= S \left( \sum_{\ell}^n M_{ij}^\ell h_\ell \right) = 
	S (h_i \otimes h_j)$ which agrees with $\mathfrak{s}_i \star \mathfrak{s}_j$ by Theorem~\ref{thmdgva1}. Now we use \eqref{starproducteqn}.
\end{proof}

The two main Theorems~\ref{thmdgva1} and~\ref{thmdgva2} can be combined as follows:
Let $\cat{C}$ be a modular category.
By Proposition~\ref{propdgva} the monoidal product induces a non-unital $E_2$-multiplication $\otimes$ on the Hochschild chain complex
supported, up to equivalence, in homological degree zero. In degree zero, however, it is relatively complicated and can be described by the $S$-transformation and the modified trace, see Theorem~\ref{thmdgva2}.
The $E_2$-multiplication on the Hochschild cochain complex
induced by the monoidal product and unimodularity (Theorem~\ref{thmverlindecochain}) behaves totally differently; it will generally have a non-trivial Gerstenhaber bracket and is unital.	
By means of the Calabi-Yau structure, we can dualize it to an $E_2$-coproduct on the Hochschild chain complex that we denote by $\Delta$,
see Remark~\ref{remproductcoproduct}. Finally, we can define the following $S$-twisted version of the modified trace
\begin{align}
\tau : \lint^{X\in\Proj\cat{C}} \cat{C}(X,X) \ra{\text{$S$-transformation}} \lint^{X\in\Proj\cat{C}} \cat{C}(X,X)\ra{\text{modified trace}} k \ . 
\end{align}
The three maps $\otimes, \Delta$ and $\tau$ combine into a closed topological conformal field theory.

\begin{theorem}\label{thmlambda}
	Let $\cat{C}$ be a modular category, then the following assignments extend in a canonical way to a closed topological conformal field theory $\Lambda_\cat{C}: \C \to \Ch$ with value $\lint^{X\in\Proj\cat{C}} \cat{C}(X,X)$ on the circle and	
	\begin{align}
	\label{eqnassignments}	\Lambda_\cat{C}\left(\begin{tikzpicture}[scale=0.4]
	\begin{pgfonlayer}{nodelayer}
		\node [style=none] (0) at (1.75, 9.75) {};
		\node [style=none] (1) at (1.75, 8.75) {};
		\node [style=none] (2) at (1.75, 8.25) {};
		\node [style=none] (3) at (1.75, 7.25) {};
		\node [style=none] (4) at (3.5, 9) {};
		\node [style=none] (5) at (3.5, 8) {};
		\node [style=none] (6) at (2.5, 6.75) {\phantom{x}};
	\end{pgfonlayer}
	\begin{pgfonlayer}{edgelayer}
		\draw [bend left=90] (0.center) to (1.center);
		\draw [bend right=90] (0.center) to (1.center);
		\draw [bend left=90] (2.center) to (3.center);
		\draw [bend right=90] (2.center) to (3.center);
		\draw [bend left=90] (4.center) to (5.center);
		\draw [bend right=90] (4.center) to (5.center);
		\draw [bend right=90, looseness=1.75] (2.center) to (1.center);
		\draw [in=180, out=0, looseness=0.50] (0.center) to (4.center);
		\draw [in=-180, out=0, looseness=0.75] (3.center) to (5.center);
	\end{pgfonlayer}
\end{tikzpicture}{\tiny:\left( \mathbb{S}^1\right)^{\sqcup 2}  \rightarrow  \mathbb{S}^1}\right) := \otimes \ , \quad 	\Lambda_\cat{C}\left(\begin{tikzpicture}[scale=0.4]
	\begin{pgfonlayer}{nodelayer}
		\node [style=none] (0) at (-5, 7.75) {};
		\node [style=none] (1) at (-5, 6.75) {};
		\node [style=none] (2) at (-5, 6.25) {};
		\node [style=none] (3) at (-5, 5.25) {};
		\node [style=none] (4) at (-6.75, 7) {};
		\node [style=none] (5) at (-6.75, 6) {};
		\node [style=none] (6) at (-6, 4.75) {\phantom{x}};
	\end{pgfonlayer}
	\begin{pgfonlayer}{edgelayer}
		\draw [bend left=90] (0.center) to (1.center);
		\draw [bend left=270] (0.center) to (1.center);
		\draw [bend left=90] (2.center) to (3.center);
		\draw [bend right=90] (2.center) to (3.center);
		\draw [bend left=90] (4.center) to (5.center);
		\draw [bend right=90] (4.center) to (5.center);
		\draw [bend right=270, looseness=1.75] (2.center) to (1.center);
		\draw [in=0, out=-180, looseness=0.50] (0.center) to (4.center);
		\draw [in=0, out=180, looseness=0.50] (3.center) to (5.center);
	\end{pgfonlayer}
\end{tikzpicture}{\tiny: \mathbb{S}^1\rightarrow \left(      \mathbb{S}^1\right)^{\sqcup 2}  }  \right) := \Delta \ , \quad \Lambda_\cat{C}\left(\begin{tikzpicture}[scale=0.4]
	\begin{pgfonlayer}{nodelayer}
		\node [style=none] (8) at (0, -0.25) {};
		\node [style=none] (9) at (0, -1) {};
		\node [style=none] (10) at (0, -1.5) {\phantom{x}};
	\end{pgfonlayer}
	\begin{pgfonlayer}{edgelayer}
		\draw [bend left=90, looseness=1.75] (8.center) to (9.center);
		\draw [bend right=90, looseness=1.75] (8.center) to (9.center);
	\end{pgfonlayer}
\end{tikzpicture} {\tiny : \mathbb{S}^1\rightarrow \emptyset}\right) := \tau \, . 
	\end{align}
\end{theorem}

\begin{proof}
	Denote by $\Phi_\cat{C}:\OC\to\Ch$ the trace field theory of $\cat{C}$ \cite{trace} and by $j : \C \to \OC$ the inclusion of the closed part of $\OC$ into $\OC$. We write $j^* \Phi_\cat{C}=\Phi_\cat{C}\circ j$ for the  restriction of $\Phi_\cat{C}$ to the closed part of $\OC$.
	The assertion follows if we can show
	$
	\Lambda_\cat{C}(\Sigma) = \left( \mfc(S^{-1})\right)^{\otimes q} \circ j^*\Phi_\cat{C}\left(\Sigma\right) \circ \mfc(S)^{\otimes p} \label{eqnfcphic2} $
	($p$ and $q$ are the number incoming and outgoing boundary components of $\Sigma$, respectively),
	where $\Sigma$ is one of the three bordisms in~\eqref{eqnassignments}
	and $\Lambda_\cat{C}(\Sigma)$ is defined as above.
	Indeed:
	If $\Sigma$ is the pair of pants, the needed statement
	is exactly Theorem~\ref{thmdgva2}, i.e.\ the Verlinde formula for Hochschild chains.
	If $\Sigma$ is the opposite pair of pants, the needed statement
	follows from Theorem~\ref{thmdgva1}, i.e.\ the Verlinde formula for Hochschild cochains, because the evaluation of $\Phi_\cat{C}$ on the reversed pair of pants is the usual $E_2$-structure on Hochschild cochains, but dualized via the Calabi-Yau structure (the latter is a part of Costello's result \cite{costellotcft}).
	Finally, if $\Sigma$ is the disk, the needed statement
	follows because the evaluation of the trace field theory 
	on the disk is induced by the evaluation of $\Phi_\cat{C}$ 
	on labeled disks, where it is given by the modified trace \cite[Theorem~4.9]{trace}. 
\end{proof}

\begin{corollary}[Partial three-dimensional extension of the differential graded modular functor]\label{corhighgenus}
	The differential graded modular functor $\mfc$ associated to a modular category $\cat{C}$ 
	extends to three-dimensional oriented bordisms of the form $\Sigma \times\mathbb{S}^1: \left( \mathbb{T}^2\right)^{\sqcup p} \to \left( \mathbb{T}^2\right)^{\sqcup q}$, where $\Sigma :  \left( \mathbb{S}^1\right)^{\sqcup p} \to \left( \mathbb{S}^1\right)^{\sqcup q}$ is a compact oriented two-dimensional bordism such that every component of $\Sigma$ has at least one incoming boundary component. 
\end{corollary}

\begin{remark}\label{remfurtherextension}
	Corollary~\ref{corhighgenus} does \emph{not} include an extension to bordisms of the form $\Sigma \times \mathbb{S}^1$ if 
	$\Sigma$ has no incoming boundary components. If we included this, we would have admitted enough bordisms such that $\mathbb{T}^2$ comes with an evaluation and a coevaluation.
	But this would imply that $H_*\mathfrak{F}_\cat{C}(\mathbb{T}^2)\cong HH_*(\cat{C})$ is a dualizable, hence finite-dimensional 
	graded vector space, and this will generally not be the case. 
	In fact, we are not aware of \emph{any} case where $\dim\, HH_*(\cat{C})<\infty$ holds in the non-semisimple situation.
\end{remark}

Corollary~\ref{corhighgenus}  offers a \emph{partial three-dimensional extension} of $\mathfrak{F}_\cat{C}$
to bordisms of the form $\Sigma \times \mathbb{S}^1$ subject to the condition that each component of $\Sigma$ has at least one incoming boundary component. 
Although the extension is not complete, it is exactly substantial enough for the \emph{dimensional reduction} $ \catf{Red}_{\mathbb{S}^1} \mathfrak{F}_\cat{C} := \mathfrak{F}_\cat{C}(\mathbb{S}^1\times-)$ to exist (where the requirements on the number of boundary components are still implicit).
Then Theorem~\ref{thmlambda} and its proof immediately imply the following compact reformulation of our results that comprises simultaneously 
the Verlinde formula for the Hochschild chains and cochains:

\begin{corollary}[Higher genus Verlinde formula]\label{cordimred}
	For any modular category $\cat{C}$,
	we find \begin{align}\label{eqndimred}\catf{Red}_{\mathbb{S}^1} \mathfrak{F}_\cat{C}\ \stackrel{S}{\simeq} \ \Phi_\cat{C}\ , \end{align}
	i.e.\
	the dimensional reduction of the partial extension of the differential graded modular functor $\mathfrak{F}_\cat{C}$ to non-invertible three-dimensional bordisms from Corollary~\ref{corhighgenus} is equivalent, via the $S$-transformation, to the trace field theory $\Phi_\cat{C}$ of $\cat{C}$ (the topological conformal field theory associated to the modified trace).
	This is an equivalence of closed topological conformal field theories.
\end{corollary}

One could na\"ively think that one could \emph{define}
the partial three-dimensional extension in a way that 
makes the above Corollary a tautology, but this does not work: The non-trivial point is that Corollary~\ref{cordimred} is not a statement about \emph{some} partial extension of $\mfc$, but \emph{the one obtained from Corollary~\ref{corhighgenus}}, for which we have given a concrete description \emph{independent} of~\eqref{eqndimred}. Hence, if one used \eqref{eqndimred} as a definition, one would still need the Verlinde formula for both Hochschild chains and cochains to arrive at Corollary~\ref{cordimred}.
As yet another caveat in connection to Corollary~\ref{cordimred}, it is important to stress that the trace field theory $\Phi_\cat{C}$ only knows about the dimensional reduction of the partial extension of $\mathfrak{F}_\cat{C}$, but has practically no information on the mapping class group actions on differential graded conformal blocks.
Instead, \eqref{eqndimred}  describes the multiplicative structures on $\mathfrak{F}_\cat{C}(\mathbb{T}^2)$ in terms of the linear category $\cat{C}$ and the modified trace which, after all, is very much in the spirit of the original Verlinde formula.

\small 
\spaceplease


\begin{thebibliography}{10000000000000}\itemsep0pt
	
	\bibitem[BK00]{bakifm}
	B. Bakalov, A. Kirillov.
	On the Lego-Teichmüller game.
	\emph{Transformation Groups} 5:207--244, 2000.
	
	
	\bibitem[BK01]{baki}
	B. Bakalov, A. Kirillov.
	\emph{Lectures on tensor categories and modular functor.}
	University Lecture Series 21,
	Amer. Math. Soc., 2001.
	
	
	\bibitem[BDSPV15]{BDSPV15}
	B. Bartlett, C. L. Douglas, C. J. Schommer-Pries, J. Vicary.
	{Modular Categories as Representations of the 3-dimensional Bordism Category}.
	arXiv:1509.06811 [math.AT]
	
	
	
	
	\bibitem[BF04]{bergerfresse} C. Berger, B. Fresse.
	Combinatorial operad actions on cochains.
	\emph{Math. Proc. Cambridge Philos. Soc.} 137(1):135--174, 2004.
	
	
	
	
	
	\bibitem[Bic13]{bichon}
	J. Bichon.
	Hochschild homology of Hopf algebras and free Yetter-Drinfeld resolutions of the counit.
	\emph{Compositio Math.} 149:658--678, 2013.
	
	\bibitem[BV12]{bruguieresvirelizier}	A. Bruguières, A. Virelizier. Quantum double of Hopf monads and categorical centers.
	\emph{Trans. Amer. Math. Soc.} 364(3):1225--1279, 2012.
	
	
	
	
	\bibitem[Car89]{cardy}
	J. Cardy.
	Boundary conditions, fusion rules and the Verlinde formula.
	\emph{Nucl. Phys. B}
324(3):581--596,
	1989.
	
	
	\bibitem[CE56]{cartaneilenberg} H. Cartan, S. Eilenberg. \emph{Homological algebra.} Princeton University Press, 1956.
	
	
	
	
	\bibitem[Cos07]{costellotcft}
	K. Costello.
	Topological conformal field theories and Calabi-Yau categories.
	\emph{Adv. Math.}
	210:165--214,
	2007.
	
	\bibitem[CGR20]{cgr}
	T. Creutzig, A. M. Gainutdinov, I. Runkel.
	A quasi-Hopf algebra for the triplet vertex operator algebra.
	\emph{Comm.  Contemp. Math.} 
	22(3), 2020.
	
	\bibitem[DMNO13]{dmno}
	A. Davydov, M. Müger, D. Nikshych,  V. Ostrik. The Witt group of non-degenerate
	braided fusion categories. \emph{J. Reine Angew. Math.} 677:135--177, 2013.
	
	\bibitem[DRGGPMR22]{gai}
	M. De Renzi, A. M. Gainutdinov, N. Geer, B. Patureau-Mirand, I. Runkel.
	3-Dimensional TQFTs from Non-Semisimple Modular Categories.
	\emph{Selecta Math. New Ser.} 28(42), 2022.
	
	
	
	\bibitem[Dri90]{drinfeld}
	V. Drinfeld.
	On almost commutative Hopf algebras.
	\emph{Leningrad Math. J.}
	1(2):321--342, 1990.
	
	
	
	
	\bibitem[ES15]{egas}
	D. Egas Santander.
	Comparing fat graph models of moduli space.
	arXiv:1508.03433 [math.AT]
	
	
	\bibitem[EGNO15]{egno}
	P. Etingof, S. Gelaki, D. Nikshych, V. Ostrik.
	\emph{Tensor categories.}
	Math. Surveys  Monogr.  205,
	Amer. Math. Soc., 2015.
	
	\bibitem[ENO04]{eno-d}
	P. Etingof, D. Nikshych,  V. Ostrik. An analogue of Radford's $S^4$
	formula for finite tensor
	categories. \emph{Int. Math. Res. Not.} 54:2915--2933, 2004.
	
	\bibitem[EO04]{etingofostrik}
	P. Etingof, V. Ostrik. {Finite tensor categories.} 
	\emph{Mosc. Math. J.}
	4(3):627--654, 2004.
	
	
	\bibitem[FS04]{farinatisolotar}
	M. A. Farinati, A. L. Solotar.
	G-structure on the cohomology of Hopf algebras.
	\emph{Proc. Amer. Math. Soc.}
	132(10):2859--2865, 2004.
	
	\bibitem[FGST06]{fgst} B. L. Feigin, A. M. Gainutdinov, A. M. Semikhatov, I. Y. Tipunin. 
	Modular group representations and fusion in logarithmic conformal field theories and in the quantum group center. \emph{Commun. Math. Phys.} 265:47--93, 2006. 
	
	
	
	
	\bibitem[FGSS18]{cardycase}
	J. Fuchs, T. Gannon, G. Schaumann, C. Schweigert.
	The logarithmic Cardy case: Boundary states and annuli.
	\emph{Nucl. Phys. B.} 930:287--327, 2018.
	
	\bibitem[FHST04]{fhst} J. Fuchs, S. Hwang, A. M. Semikhatov, I. Y. Tipunin. Nonsemisimple fusion algebras and the Verlinde formula. \emph{Commun. Math. Phys.}
	247:713--742, 2004. 
	
	
	\bibitem[FSS20]{fss}
	J. Fuchs, G. Schaumann, C. Schweigert.
	{Eilenberg-Watts calculus for finite categories and a bimodule Radford $S^4$ theorem}.
	\emph{Trans. Amer. Math. Soc.} 373:1--40, 2020.
	
	
	
	
	\bibitem[FSS14]{fuchsschweigertstigner}
	J. Fuchs, C. Schweigert, C. Stigner.
	Higher genus mapping class group invariants from factorizable Hopf algebras.
	\emph{Adv. Math.}
	250:285--319,
	2014.
	
	
	\bibitem[GLO18]{glo}
	A. M. Gainutdinov, S. Lentner, T. Ohrmann.
	Modularization of small quantum groups.
	arXiv:1809.02116 [math.QA]
	
	
	\bibitem[GR19]{grv}
	A. M. Gainutdinov, I. Runkel.
	The non-semisimple Verlinde formula and pseudo-trace functions.
	\emph{J. Pure App. Alg.} 223(2):660--690, 2019.
	
	
	\bibitem[GR20]{gainutdinovrunkel}
	A. M. Gainutdinov, I. Runkel.
	Projective objects and the modified trace in factorisable finite tensor categories.
	\emph{Compositio Math.} 156:770--821, 2020.
	
	\bibitem[GT09]{gainutdinovtipunin}
	A. M. Gainutdinov, I. Y. Tipunin. Radford, Drinfeld, and Cardy boundary states in $(1,p)$ logarithmic conformal field models. \emph{J. Phys. A} 42, 2009. 
	
	\bibitem[GKP11]{mtrace1} N. Geer, J. Kujawa, B. Patureau-Mirand. Generalized trace
	and modified dimension functions on ribbon categories. \emph{Selecta Math. New Ser.}
	17(2):453--504, 2011. 
	
	
	\bibitem[GKP13]{mtrace2}  N. Geer, J. Kujawa, B. Patureau-Mirand. Ambidextrous objects and trace functions for nonsemisimple categories. \emph{Proc. Amer. Math. Soc.}
	141(9):2963--2978, 2013.
	
	
	
	\bibitem[GKP22]{mtrace}
	N. Geer, J. Kujawa, B. Patureau-Mirand. 
	M-traces in (non-unimodular) pivotal categories.
	\emph{Algebr. Represent. Theory} 25:759--776, 2022
	
	
	\bibitem[GPT09]{geerpmturaev}	N. Geer, B. Patureau-Mirand, V. Turaev. Modified Quantum Dimensions and re-normalized Link Invariants. \emph{Compositio Math.} 145:196--212, 2009.
	
	\bibitem[GPV13]{mtrace3} N. Geer, B. Patureau-Mirand, A. Virelizier. Traces on ideals in
	pivotal categories. \emph{Quantum Topol.} 4(1):91--124, 2013.
	
	\bibitem[Ger63]{gerstenhaber}
	M. Gerstenhaber.
	The cohomology structure of an associative ring.
	\emph{Ann. Math.} 78:59--103, 1963.
	
	
	
	
	
	\bibitem[GK93]{gk}
	V. Ginzburg, S. Kumar.
	Cohomology of quantum groups at roots of unity.
	\emph{Duke Math. J.}
	69(1):179--198, 1993.
	
	\bibitem[Gro84]{grothendieck}
	A. Grothendieck.
	Esquisse d'un Programme.
	CNRS research proposal from 1984.
	Printed in: \emph{Geometric Galois Actions. Volume 1: Around Grothendieck's Esquisse d'un Programme.}
	Edited by L. Schneps, P. Lochak. Cambridge University Press, 1997.
	
	\bibitem[Har83]{harer}   J. Harer. The second homology group of the mapping class group of an orientable surface.
	\emph{Invent. Math.}
	72:221--239, 1983.
	
	
	
	\bibitem[HT80]{hatcherthurston}
	A. Hatcher, W. Thurston.
	A presentation for the mapping class group of a closed orientable surface. \emph{Topology}
	19:221--237, 1980.
	
	\bibitem[Her16]{hermann}
	R. Hermann.
	Monoidal categories and the Gerstenhaber bracket in Hochschild cohomology.
	\emph{Mem. Amer. Math. Soc.} 243(1151), 2016.
	
	
	\bibitem[Hua08a]{huangverlinde}
	Y.-Z. Huang.
	Vertex operator algebras and the Verlinde Conjecture.
	\emph{Commun. Contemp. Math.} 10(01):103--154, 2008.
	
	\bibitem[Hua08b]{huang} Y.-Z. Huang. Rigidity and modularity of vertex tensor categories. \emph{Commun. Contemp. Math.}
	10(1):871--911, 2008.
	
	
	
	
	
	\bibitem[Kar19]{eilind}
	E. Karlsson.
	Bulk fields in CFT beyond rational CFT.
	Master Thesis, University of Hamburg, 2019.
	
	
	
	
	\bibitem[Kas95]{kassel}
	C. Kassel.
	\emph{Quantum Groups}.
	Springer-Verlag New York, 1995.
	
	\bibitem[Kau08]{kaufmann}
	R. M.	Kaufmann.
	A proof of a cyclic version of Deligne's conjecture via Cacti.
	\emph{Math. Res. Letters} 15(5):901--921, 2008.
	
	
	\bibitem[KLM01]{klm}
	Y. Kawahigashi, R. Longo, M. M\" uger.
	Multi-interval Subfactors and Modularity of Representations in Conformal Field Theory.
	\emph{Commun. Math. Phys.} 219:631--669, 2001.
	
	\bibitem[Kel99]{keller}
	B. Keller.
	{On the cyclic homology of exact categories}.
	\emph{J. Pure Appl. Alg.} 136:1--56, 1999.
	
	
	\bibitem[KL01]{kl}
	T. Kerler, V. V. Lyubashenko. \emph{Non-Semisimple
		Topological Quantum Field
		Theories for 3-Manifolds
		with Corners}.
	Lecture Notes in Mathematics 1765, Springer-Verlag Berlin Heidelberg, 2001.
	
	
	
	
	
	\bibitem[LQ21]{lq}
	A. Lachowska, Y. Qi.
	Remarks on the derived center of small quantum groups.	
	\emph{Selecta Math. New Ser.} 27:68, 2021.
	
	
	
	
	
	\bibitem[LMSS18]{svea}
	S. Lentner, S. N. Mierach, C. Schweigert, Y. Sommerhäuser.
	{Hochschild Cohomology and the Modular Group.}
	\emph{J. Alg.}
	507:400--420, 2018.
	
	\bibitem[LMSS20]{svea2}
	S. Lentner, S. N. Mierach, C. Schweigert, Y. Sommerhäuser.
	{Hochschild Cohomology, Modular Tensor Categories, and Mapping Class Groups.} 	Accepted by \emph{Springer Briefs in Math. Phys.}
	arXiv:2003.06527 [math.QA]
	
	\bibitem[LO17]{lo}
	S. Lentner, T. Ohrmann.
	Factorizable $R$-Matrices for Small Quantum Groups.
	\emph{SIGMA} 13(076):1--25, 2017.
	
	\bibitem[LZ14]{lezhou}
	J. Le, G. Zhou. On the Hochschild cohomology ring of tensor products of algebras.
	\emph{J. Pure Appl. Alg.} 218(8):1463--1477, 2014.
	
	
	
	
	\bibitem[Lyu95a]{lyubacmp}
	V. V. Lyubashenko. {Invariants of 3-manifolds and projective representations of mapping class groups  via  quantum  groups  at  roots  of  unity.} \emph{Commun. Math. Phys.} 172:467--516, 1995.
	
	
	\bibitem[Lyu95b]{lyuba}
	V. V. Lyubashenko. {Modular transformations for tensor categories.}
	\emph{J. Pure Appl. Alg.} 98(3):279--327, 1995.
	
	\bibitem[Lyu96]{lyubalex}
	V. V. Lyubashenko. {Ribbon abelian categories as modular categories.}
	\emph{J. Knot Theory and its
		Ramif.} 5:311--403, 1996.
	
	
	
	\bibitem[MPSW09]{schau}
	M. Mastnak, J. Pevtsova, P. Schauenburg, S. Witherspoon.
	Cohomology of finite dimensional pointed Hopf algebras.
	\emph{Proc. London Math. Soc.}  100(2):377--404, 2009.
	
	
	\bibitem[MCar94]{mcarthy}
	R. McCarthy. {The cyclic homology of an exact category.} \emph{J. Pure Appl. Alg.} 93:251--296, 1994.
	
	
	\bibitem[MCS02]{cluresmith}
	J. E. McClure, J. H. Smith. A solution of Deligne’s Hochschild cohomology conjecture.
	In: \emph{Recent progress in homotopy theory}. 
	\emph{Contemp.
		Math.} 293:153--193,  2002.
	
	\bibitem[Men11]{menichi} L. Menichi.
	Connes-Moscovici characteristic map is a Lie algebra morphism.
	\emph{J. Alg.}
	331(1):311--337, 2011.
	
	\bibitem[MS88]{mose88}
	G. Moore, N. Seiberg.
	Polynomial equations for rational conformal field theories.
	\emph{Phys. Letters B}
	212(4):451--460, 1988.
	
	
	
	\bibitem[MS90]{mooreseiberg}
	G. Moore, N. Seiberg. 
	Lectures on RCFT. In: Physics, Geometry, and Topology (ed. by H.C.
	Lee), 263--361. Plenum Press, New York 1990.
	
	
	
	\bibitem[NP22]{negronplavnik}
	C. Negron, J. Y. Plavnik.
	Cohomology of finite tensor categories: duality and Drinfeld centers
	\emph{Trans. Amer. Math. Soc.} 375:2069--2112, 2022.
	
	
	
	\bibitem[NS98]{neuchlschauenburg}
	M. Neuchl, P. Schauenburg.
	Reconstruction in braided categories and a notion of
	commutative bialgebra.
	\emph{J. Pure Applied Alg.} 124:241--259, 1998.
	
	
	
	\bibitem[NTV03]{ntv}
	D. Nikshych, V. Turaev, L. Vainerman.
	Invariants of knots and 3-manifolds from quantum groupoids.
	\emph{Top. App.} 127(1-2):91--123, 2003.
	
	
	
	\bibitem[RT90]{rt1} N. Reshetikhin, V. Turaev. {Ribbon graphs and their invariants derived from
		quantum groups.} \emph{Comm. Math. Phys.} 127:1--26, 1990.
	
	\bibitem[RT91]{rt2}
	N. Reshetikhin, V. Turaev. {Invariants of 3-manifolds via link polynomials and quantum groups.} \emph{Invent.
		Math.} 103:547--597, 1991.
	
	
	\bibitem[SW03]{salvatorewahl}
	P. Salvatore, N. Wahl.
	{Framed discs operads and Batalin-Vilkovisky algebras}. 
	\emph{Quart. J. Math.} 54:213--231, 2003. 
	
	\bibitem[SW21a]{dva}C. Schweigert, L. Woike.
	{The Hochschild Complex of a Finite Tensor Category.} \emph{Alg. Geom. Top.} 21(7):3689--3734, 2021.
	
	
	
	\bibitem[SW21b]{dmf}C. Schweigert, L. Woike.
	Homotopy Coherent Mapping Class Group Actions and Excision for Hochschild Complexes of Modular Categories.
	\emph{Adv. Math.} 386, 107814, 2021.
	
	
	\bibitem[SW22a]{trace}
	C. Schweigert, L. Woike.
	The Trace Field Theory of a Finite Tensor Category.
	\emph{Algebr. Represent. Theory} (online first), 2022.
	
	\bibitem[SW22b]{E2}
	C. Schweigert, L. Woike. Homotopy Invariants of Braided Commutative Algebras and the Deligne Conjecture for Finite Tensor Categories. arXiv:2204.09018 [math.QA]
	
	
	\bibitem[SS21]{ss21}
	T. Shibata, K. Shimizu.
	Modified traces and the Nakayama functor.
	\emph{Algebr. Represent. Theory} (online first), 2021.
	
	\bibitem[Shi17a]{shimizuunimodular}
	K. Shimizu.
	{On unimodular finite tensor categories}. 
	\emph{Int. Math. Res. Not.} 2017(1):277--322, 2017
	
	
	
	\bibitem[Shi17b]{shimizucf}	K. Shimizu. The monoidal center and the character algebra. \emph{J. Pure Appl. Alg.} 221(9):2338--2371, 2017.
	
	
	\bibitem[Shi19]{shimizumodular}
	K. Shimizu.
	{Non-degeneracy conditions for braided finite tensor categories}. \emph{Adv. Math.} 355, 106778, 2019.
	
	
	
	\bibitem[Shi20]{shimizucoend}
	K. Shimizu.
	Further Results on the Structure of (Co)Ends in Finite Tensor Categories.
	\emph{App. Cat. Str.} 28:237--286, 2020.
	
	
	
	
	\bibitem[Tam98]{tamarkin}
	D. E. Tamarkin.
	Another proof of M. Kontsevich formality theorem.
	arXiv:math/9803025 [math.QA]
	
	
	\bibitem[Til98]{tillmann}
	U. Tillmann.
	$\mathcal{S}$-Structures for $k$-Linear Categories and the Definition of a Modular Functor.
	\emph{J. London Math. Soc.}
	58(1):208--228, 1998.
	
	
	\bibitem[TZ06]{tradlerzeinalian}
	T. Tradler, M. Zeinalian. On the cyclic Deligne conjecture. 
	\emph{J. Pure App. Alg.}
	204(2):280--299, 2006.
	
	
	
	
	
	
	\bibitem[Tur94]{turaev}
	V.  Turaev. \emph{Quantum Invariants of Knots and 3-Manifolds.}
	Studies in Mathematics 18, De Gruyter. First edition: 1994, second edition: 2010.
	
	
	\bibitem[Ver88]{verlinde}
	E. Verlinde.
	Fusion rules and modular transformations in 2D conformal field
	theory. \emph{Nucl. Phys. B} 300:360--376, 1988.
	
	
	
	\bibitem[WW16]{wahlwesterland}
	N. Wahl, C. Westerland.
	Hochschild homology of structured algebras.
	\emph{Adv. Math.} 288:240--307, 2016.
	
	
	
	
	
	\bibitem[Wit89]{witten} E. Witten.
	Quantum field theory and the Jones polynomial.
	\emph{Comm. Math. Phys.}
	121(3):351--399, 1989.
	
	
	
	
	
	
	
	
\end{thebibliography}
\end{document}